\documentclass{aims}
\usepackage{a4wide}
\usepackage{amsmath}
\usepackage{paralist}
\usepackage{eucal}
\usepackage{mathrsfs}
\usepackage[colorlinks=true]{hyperref}
\hypersetup{linkcolor=blue, urlcolor=blue, citecolor=red}

\def\currentvolume{X}
 \def\currentissue{X}
  \def\currentyear{200X}
   \def\currentmonth{XX}
    \def\ppages{X--XX}
     \def\DOI{10.3934/xx.xx.xx.xx}

\newtheorem{theorem}{Theorem}[section]
\newtheorem{corollary}[theorem]{Corollary}

\newtheorem{lemma}[theorem]{Lemma}
\newtheorem{proposition}[theorem]{Proposition}

\theoremstyle{definition}
\newtheorem{definition}[theorem]{Definition}
\newtheorem{remark}{Remark}

\def\GGG{} 
\def\nc{\normalcolor}

\newcommand{\RCD}{\mathrm{RCD}}

\def\erre{\mathbb{R}}
\def\cP{\mathcal{P}}

\def\eps{\varepsilon}
\def\tD{\textnormal{D}}
\def\gm{\mathfrak{m}}
\def\epy{\KL}
\def\tnLip{\textnormal{Lip}}
\def\sfP{\mathsf{P}}
\def\sfP{\mathsf{P}}
\def\mcD{\mathcal{D}}
\def\hed{\mathsf{He}}
\def\dd{\textnormal{d}}
\def\beq{\begin{equation}}
\def\eeq{\end{equation}}

\def\to{\rightarrow}

\def\norm #1{\left\|#1\right\|}
\def\abs #1{\left\vert#1\right\vert}

\newcommand{\HK}{\mathsf{H\kern-3pt K}}
\newcommand{\KL}{\mathsf{K\kern-1.5pt L}}

\newcommand{\sfd}{\mathsf d}
\newcommand{\sfW}{\mathsf W}
\newcommand{\calD}{\mathcal D}
\newcommand{\calM}{\mathcal M}
\newcommand{\calP}{\mathcal P}
\newcommand{\rmW}{\mathrm W}
\newcommand{\rmC}{\mathrm C}
\newcommand{\ggamma}{\boldsymbol \gamma}

\title[Heat flows in metric measure spaces] 
      {Contraction and regularizing properties
		 of heat flows in metric measure spaces}

\author[Giulia Luise and Giuseppe Savar\'e]{}



 \email{g.luise.16@ucl.ac.uk}
 \email{giuseppe.savare@unipv.it}

\thanks{The second author is partially supported by PRIN2015 grant from MIUR for
  the project
  \emph{Calculus of Variations} and by IMATI-CNR}

\begin{document}
\maketitle

\centerline{\scshape Giulia Luise}
\medskip
{\footnotesize
  \centerline{Department of Computer Science, University College London}
  \centerline{Gower Street, London WC1E 6BT, UK}
} 

\medskip

\centerline{\scshape Giuseppe Savar\'e $^*$}
\medskip
{\footnotesize
 \centerline{Dipartimento di Matematica `Felice Casorati', University of Pavia}
   \centerline{Via Ferrata 1, 27100 Pavia, Italy}
}
\medskip
\centerline{\emph{Dedicated to Alexander Mielke on the occasion of his
  60th birthday}}

\bigskip


\begin{abstract}\noindent  We illustrate some novel contraction and
  regularizing properties of the Heat flow in metric-measure spaces
  that emphasize an interplay between Hellinger-Kakutani,
  Kantorovich-Wasserstein
  and Hellinger-Kantorvich distances.
  Contraction properties of
  Hellinger-Kakutani distances and general Csisz\'ar
  divergences hold in arbitrary metric-measure spaces and do not
  require assumptions on the linearity of the flow.
  
  When weaker transport distances are involved, we will
  show that contraction and regularizing effects rely on the dual
  formulations of the distances and are strictly related to lower
  Ricci curvature bounds 
  in the setting of $\RCD(K,\infty)$ metric measure spaces.  
\end{abstract}



\thispagestyle{empty}
\tableofcontents

\section{Introduction}
The study of contraction properties of $L^p$ norms and
more general convex entropy
functionals with
respect to the action 
of Markov semigroups is a very classic subject
(see e.g.~\cite{Bakry-Gentil-Ledoux14}).
More recently,
the role of the
Kantorovich-Rubinstein-Wasserstein metric $\sfW_2$
for second order diffusion equations 
in the space of
probability measures
has been deeply investigated, starting from the pioneering contribution by F.~Otto
\cite{Otto01}.
Many investigations have clarified the relations between
analytic estimates depending on the structure of the generating
differential operator
and geometric properties of the underlying spaces,
with an increasing level of generality.
An incomplete list of contributions includes the contraction of a general
class of evolution equations combining diffusion, interaction and drift
\cite{Carrillo-McCann-Villani06}, the gradient-flow structure and the
geodesic convexity in Euclidean spaces
\cite{JordanKinderlehrerOtto98,Otto01,AGS08},
the Heat flow in Riemannian manifolds and the Ricci curvature
\cite{Otto-Villani00,Otto-Westdickenberg05,Sturm-VonRenesse05,Daneri-Savare08,Erbar10,Villani09},
Hilbert geometry \cite{Ambrosio-Savare-Zambotti09},
the duality with gradient estimates and the Alexandrov spaces
\cite{Kuwada10,Gigli-Kuwada-Ohta13},
the RCD metric measure spaces and
the Bakry-\'Emery condition
\cite{AGS14I,AGS14D,AGS15,BGL15,Erbar-Kuwada-Sturm15,AMS15preprint}.

In one of the most general formulations, we will deal with 
a metric-measure space $(X,\sfd,\gm)$
given by a complete and separable metric space $(X,\sfd)$
endowed with a Borel positive measure $\gm$
with full support
satisfying the growth condition
  \begin{equation}
    \exists o\in X,\ \kappa\ge 0:\quad
    \gm(\{ x: \,\,\mathsf{d}(x,o)<r
    \}) \le \textnormal{e}^{\kappa r^2}.
    \label{eq:80}
\end{equation}
We introduce the Cheeger energy functional
$\mathsf{Ch}:L^2(X,\gm)\to [0,+\infty]$ 
\begin{equation}
  \label{eq:54}
  \mathsf{Ch}(f):=\inf\Big\{
  \liminf_{n\to\infty}\frac 12\int_X |\mathrm D f_n|^2\,\dd\gm,\
  f_n\in 
  \tnLip_b(X),\ f_n\to f\text{ in }L^2(X,\gm)\Big\}
\end{equation}
where
\begin{equation}
  \label{eq:55}
  |\mathrm D f|(x):=\limsup_{y\to
    x}\frac{|f(y)-f(x)|}{\sfd(x,y)};\quad
  |\mathrm D f|(x):=0\text{ if $x$ is isolated.}
\end{equation}
$\mathsf{Ch}$ is a convex, $2$-homogeneous and lower semicontinuous
functional 
whose proper domain $\calD(\mathsf{Ch})=
\{f\in L^2(X,\gm):\mathsf{Ch}(f)<\infty\}$
provides one of the equivalent characterization of the
metric Sobolev space $\rmW^{1,2}(X,\sfd,\gm)$
(see also
\cite{Heinonen-Koskela98,Koskela-MacManus98,Shanmugalingam00,Bjorn-Bjorn11,Heinonen-Koskela-Shanmugalingam-Tyson15}). A
local
weak gradient $|\mathrm D f|_w\in L^2(X,\gm)$ can be associated to
each function $f\in \rmW^{1,2}(X,\sfd,\gm)$ so that
the Cheeger energy admits the integral representation
\begin{displaymath}
  \mathsf{Ch}(f)=\frac 12\int_X |\mathrm D f|_w^2(x)\,\dd\gm(x).
\end{displaymath}
The $L^2$ subdifferential of $\mathsf {Ch}$ (whose
minimal selection will be denoted by $-\Delta$)
generates  a continuous
semigroup of order preserving contractions $(\sfP_t)_{t\ge0}$
in $L^2(X,\gm)$, which is canonically attached to the
metric-measure structure $(X,\sfd,\gm)$.

Even if in general the operators $\sfP_t$ are not linear, one can prove
\cite{AGS14I} that the semigroup is contractive with respect to all the $L^p$
norms, $p\in [1,+\infty]$,
\begin{equation}
  \label{eq:58}
  \|\sfP_t f-\sfP_t g\|_{L^p(X,\gm)}\le \|f-g\|_{L^p(X,\gm)}\quad
  \text{for every $f,g\in L^2\cap L^p(X,\gm)$},
\end{equation}
and all the integral functionals with convex integrand $\phi:\mathbb
R\to[0,+\infty)$
\begin{equation}
  \label{eq:59}
  \int_X \phi(\sfP_t f)\,\dd\gm\le \int_X \phi(f)\,\dd\gm\quad
  \text{for every }f\in L^2(X,\gm).
\end{equation}
A first important result we will prove in Section
\ref{sec:contraction} is the extension of \eqref{eq:58}-\eqref{eq:59}
to
arbitrary convex integral
functionals 
on evolving pairs:
\begin{equation}
  \label{eq:60}
  \int_X E(\sfP_t f,\sfP_t g)\,\dd\gm\le \int_X E(f,g)\,\dd\gm\quad
  \text{for every $f,g\in L^2(X,\gm)$},
\end{equation}
whenever $E:\erre^2\to [0,+\infty]$ is a lower semicontinuous convex
integrand with $E(0,0)=0$.
As a byproduct, we obtain that the action of
$(\sfP_t)_{t\ge0}$ on
nonnegative functions $f,g\in L^1(X,\gm)$
is a contraction
with respect to arbitrary Csisz\'ar divergences
(see \cite{Csiszar67,Liese-Vajda06} and Section \ref{sec:preliminaries}),
such as
the Kullback-Leibler entropy functional
\cite{Kullback-Leibler51}
associated to $E(r,s)=r\ln (r/s)-r+s$ if $r,s>0$, yielding
(since $\sfP_t$ is mass preserving)
\begin{displaymath}
  \int_{\sfP_t g>0} \ln\Big(\sfP_t f/\sfP_t g\Big)\sfP_t f\,\dd\gm\le
  \int_{g>0} \ln\Big(f/g\Big) f\,\dd\gm,
\end{displaymath}
or the Hellinger-Kakutani distances
\cite{Hellinger09,Kakutani48}
\begin{displaymath}
  \int_X |(\sfP_t f)^{1/p}-(\sfP_t g)^{1/p}|^p\,\dd\gm\le
  \int_X |f^{1/p}-g^{1/p}|^p\,\dd\gm\quad
  p\in [1,+\infty),
\end{displaymath}
associated to $E(r,s)=|r^{1/p}-s^{1/p}|^p$, $r,s\ge0$

The most relevant connections with optimal transport metrics
occur when $\mathsf{Ch}$ is also a quadratic form, i.e.~it satisfies the
parallelogram rule
\begin{equation}
  \label{eq:63}
  \mathsf{Ch}(f+g)+
  \mathsf{Ch}(f-g)=2
  \mathsf{Ch}(f)+2   \mathsf{Ch}(g),\quad
  \text{for every }f,g\in \calD(\mathsf{Ch}).
\end{equation}
In this case
$-\Delta$ is a linear positive selfadjoint operator in $L^2(X,\gm)$
and $(\mathsf P_t)_{t\ge0}$ is a linear Markov semigroup associated to a strongly local symmetric Dirichlet form
$\mathcal E$ on $L^2(X,\mathfrak m)$,
admitting Carr\'e du Champ
$\Gamma:\calD(\mathsf {Ch})\times \calD(\mathsf{Ch})\to L^1(X,\gm)$
which provides a bilinear extension of the weak gradient, since
\begin{displaymath}
  \Gamma(f,f)=|\mathrm D f|^2_w\quad \text{for every }f\in \rmW^{1,2}(X,\sfd,\gm).
\end{displaymath}
If every bounded function $f\in \rmW^{1,2}(X,\sfd,\gm)$ 
with $|\mathrm Df|_w\le 1$ $\mathfrak m$-a.e.~admits a $\mathsf
d$-continuous representative
(still denoted by $f$) which satisfies the $1$-Lipschitz condition
\begin{displaymath}
  |f(y)-f(x)|\le \sfd(x,y)\quad
  \text{for every }x,y\in X
\end{displaymath}
then $\Delta$ satisfies (a suitable weak formulation of) the
Bakry-\'Emery condition $\mathrm{BE}(K,\infty)$, $K\in \mathbb R,$
\begin{equation}
  \label{eq:savare3}
  \Gamma_2(f)=\frac 12 \Delta\Gamma(f,f)-\Gamma(f,\Delta f)\ge
  K\,\Gamma(f)
\end{equation}
if and only
if
$(\mathsf P_t)_{t\ge0}$ admits a (unique) extension
$(\sfP_t^*)_{t\ge0}$  to the space of finite
Borel measures $\mathcal M(X)$
and satisfies  the contraction property (see \cite{AGS15})
\begin{equation}
  \label{eq:64}
  \sfW_2(\mathsf P_t^* \mu_0,\mathsf P_t^*\mu_1)\le \mathrm
  e^{-Kt}\sfW_2(\mu_0,\mu_1)\quad
  \text{for every }\mu_0,\mu_1\in \mathcal P_2(X);
\end{equation}
here $\sfW_2$ denotes the $2$-Kantorovich-Wasserstein distance
between probability measures of $\mathcal P_2(X)$
with finite quadratic moments
\begin{displaymath}
  \begin{aligned}
    \sfW_2^2(\mu_0,\mu_1):=\min\Big\{&\int_{X\times
      X}\mathsf{d}^{2}(x_0,x_1)\,\dd \boldsymbol{\mu}(x_0,x_1):
    \boldsymbol{\mu}\in\mathcal P(X\times X),\\
    &\ \pi^0_\sharp     \boldsymbol{\mu}=\mu_0,\ \pi^1_\sharp
        \boldsymbol{\mu}=\mu_1\Big\},\quad
        \pi^i(x_0,x_1):=x_i,\ i=0,1.
  \end{aligned}
\end{displaymath}
In fact, this property is deeply related with the synthetic
theory of $\mathsf{CD}(K,\infty)$ metric-measure spaces with Ricci curvature
bounded from below developed by Lott-Villani \cite{Lott-Villani09} and
Sturm
\cite{Sturm06I}. The combination of the Lott-Sturm-Villani condition
$\mathsf{CD}(K,\infty)$ with the
quadratic property of the Cheeger energy \eqref{eq:63}
provides one of the equivalent characterizations of the so-called
$\RCD(K,\infty)$ metric-measure space \cite{AGS14D}, which
turned out to be equivalent with the Bakry-\'Emery
functional-analytic approach
we have adopted here \cite{AGS15}.

The link between \eqref{eq:savare3} and \eqref{eq:64} becomes more
apparent if we consider that \eqref{eq:savare3} is in fact equivalent
to the Bakry-\'Emery commutation estimate 
\begin{equation}
  \label{eq:48}
  |\mathrm D\sfP_t f|^2 \le \mathrm e^{-2K t}
  \sfP_t\big(|\mathrm D f|^2\big)
  \quad \text{for every }f\in \tnLip_b(X),
\end{equation}
combined with
the duality formula expressing the distance $\sfW_2$ in terms of regular
subsolutions 
$\zeta\in \mathrm C^1([0,1];\mathrm{Lip}_b(X))$ 
to the Hamilton-Jacobi equation \cite{Otto-Villani00,AGS14I,AES16}
\begin{equation}
  \label{eq:69}
  \frac 12 \sfW_2^2(\mu_0,\mu_1)=
  \sup\Big\{
  \int_X\zeta_1\,\mathrm d\mu_1-
  \int_X\zeta_0\,\mathrm d\mu_0: \partial_t \zeta_t+\frac 12|\mathrm D\zeta_t|^2\le 0\Big\},
\end{equation}
thanks to the dual representation formula for $\sfP_t^*$:
\begin{equation}
  \label{eq:51}
  \int_X f\,\dd(\sfP_t^* \mu)=
  \int_X \sfP_t f\,\dd \mu\quad
  \text{for every }f\in \mathrm C_b(X),\ \mu\in \mathcal M(X).
\end{equation}
\eqref{eq:48} shows in fact that $(\sfP_t)_{t\ge0}$ preserves
(up to an exponential factor)
subsolutions to the Hamilton-Jacobi equation \eqref{eq:69}.

In Section \ref{sec:regularization} we improve \eqref{eq:64} 
in two directions.
First of all, we will show that
after a strictly positive time $\sfP_t$ exhibits a regularizing
effect, providing a control of the stronger $2$-Hellinger distance
\begin{displaymath}
  \hed_2^2(\mu_0,\mu_1):=\int_X \Big(\sqrt\varrho_1-\sqrt
  \varrho_0\Big)^2\,\mathrm d\mu,\quad 
  \mu_i=\varrho_i\mu,
\end{displaymath}
in terms of the weaker Wasserstein distance between the
initial measures:
\begin{equation}
  \label{eq:65}
  \hed_2(\sfP_t^*\mu_0,\sfP_t^*\mu_1)\le
  \frac{1}{2\sqrt{R_K(t)}}\sfW_2(\mu_0,\mu_1)
  \quad\text{for every }\mu_0,\mu_1\in \mathcal P_2(X)
\end{equation}
where
\begin{equation}
  \label{eq:66}
  R_K(t)
  :=\begin{cases}
\displaystyle
\frac{\textnormal{e}^{2Kt}-1}{K}\qquad&\text{if }K\neq 0\\
2t &\text{if }K=0.
\end{cases}
\end{equation}
Notice that when $\gm\in \mathcal P_2(X)$ and $K\ge0$
we obtain the asymptotic estimate
\begin{displaymath}
  \hed_2(\sfP_t^*\mu_0,\gm)\le
  \frac{1}{2\sqrt{R_K(t)}}\sfW_2(\mu_0,\gm),
\end{displaymath}
proving in particular Hellinger convergence of $\sfP_t \mu_0$
to $\gm$ as $t\to\infty$, with exponential rate if $K>0$.

A second and more refined estimate involves the recently introduced
family of Hellinger-Kantorovich distances $\HK_{\alpha}$,
$\alpha>0$,
\cite{CPSV18,CPSV18bis,KMV16,LMS16,LMS18}, 
which can be defined in terms of an
Optimal Entropy--Transport problem
\cite{LMS16,LMS18}
\begin{displaymath}
  \HK_{\alpha}^2(\mu_0,\mu_1):=
  \min_{\boldsymbol\gamma\in \mathcal M(X\times X)}
  \KL(\gamma_0|\mu_0)+
  \KL(\gamma_1|\mu_1)+
  \int_{X\times X} \ell_{\alpha}(\sfd(x_0,x_1))\,\mathrm d\boldsymbol \gamma,
\end{displaymath}
where $\gamma_0,\gamma_1$ are the marginals of $\boldsymbol \gamma$, 
$\KL$ is the Kullback-Leibler divergence
\begin{displaymath}
  \KL(\gamma|\mu):=\int_X
  \Big(\varrho\log\varrho-\varrho+1\Big)\,\mathrm d\mu,\quad 
  \gamma=\varrho\mu\ll\mu,
\end{displaymath}
and $\ell_{\alpha}$ is the cost function
\begin{equation}\label{eq:76}
  \ell_{\alpha}(r):=
  \begin{cases}
    \log\Big(1+\tan^2\big(r/\sqrt\alpha\big)&\text{if }\mathsf
    d(x_0,x_1)<\sqrt{\alpha}\pi/2,\\
    +\infty&\text{otherwise}.
  \end{cases}
\end{equation}
It turns out that $\HK_\alpha$ (corresponding to $\HK_{\alpha,4}$
in the more general notation of \cite{LMS16,LMS18}) admits a dual dynamic representation
formula \cite{LMS18}
\begin{displaymath}
  \HK_{\alpha}^2(\mu_0,\mu_1)=
  \sup\Big\{
  \int \zeta_1\,\mathrm d\mu_1-
  \int\zeta_0\,\mathrm d\mu_0: \partial_t \zeta_t+\frac \alpha4|\mathrm D\zeta_t|^2+\zeta_t^2\le 0\Big\},
\end{displaymath}
so that when the Bakry-\'Emery condition $\mathrm{BE}(0,\infty)$ holds one has \cite{LMS18}
  \begin{displaymath}
    \HK_{\alpha}(\mathsf P_t\mu_0,\mathsf P_t\mu_1)\le 
    \HK_{\alpha}(\mu_0,\mu_1)\quad\text{for every }\mu_0,\mu_1\in
    \mathcal M(X).
  \end{displaymath}
Actually, the stronger
Hellinger distance at time $t>0$ can be estimated in terms of the weaker Hellinger-Kantorovich
one: for every $t>0$
\begin{equation}\label{stima_imp_HK}
  \hed_2(\sfP_{t}^{\ast}\mu_0,\sfP_{t}^{\ast}\mu_1)\leq
  \HK_{\alpha(t)}(\mu_0,\mu_1)\quad
  \text{where $\alpha(t)=4R_K(t)$}.
 \end{equation}
 Differently from other well known properties, the estimates
\eqref{eq:65} and \eqref{stima_imp_HK} cannot be deduced by  a regularization effect on a
single initial datum, since $\hed_2$, $\sfW_2$ and $\HK_\alpha$ are not
translation invariant. In this respect, the dual dynamic approach plays a crucial role.

\textit{Plan of the paper.}
The paper is organized as follows: in Section 2 we will
collect a few preliminary results on Csisz\'ar divergences,
Hellinger-Kakutani, Kantorovich-Wasserstein and Hellinger-Kantorovich
metrics.

Section 3 is devoted to a short review of the main tools of calculus
in metric-measure spaces, which are used throughout the work.
A brief description of the main properties of $\RCD(K,\infty)$ metric
measures spaces is also presented.

The last two sections contain novel results.
Section 4 is dedicated to the proof of \eqref{eq:60} in general metric
measure spaces.
Section 5 discusses the regularization estimates
\eqref{eq:65} and \eqref{stima_imp_HK}.
\newpage
\section{Distances and entropies on the space of finite measures}
\label{sec:preliminaries}

\subsection{Csisz\'ar divergences/Relative entropies}
\label{sub:entropy}
  We first recall a few basic facts on convex and $1$-homoge\-neous
  functionals of positive measures. 

  Let $(\Omega,\mathscr B)$ be a measurable space. We will denote the
  space of finite
  nonnegative measures on $(\Omega,\mathscr B)$ by
  $\calM(\Omega)$. 
  If $\mu_0,\mu_1\in \calM(\Omega)$,
  we say that $\lambda\in \calM(\Omega)$ is a \emph{common
    dominating measure} if $\mu_i\ll\lambda$, $i=0,1$. 
  Such a $\lambda$ always
  exists, for instance we may take $\lambda=\mu_0+\mu_1$.
  We will also often consider the Lebesgue decomposition of $\mu_0$
  w.r.t.~$\mu_1$ given by
  \begin{equation}
    \label{eq:Lebesgue}
    \mu_0=\varrho\mu_1 + \mu_0^{\bot},
    \quad \mu_0^{\bot}\bot \,\mu_1, \quad
     \varrho:=\frac{\dd\mu_0}{\dd\mu_1}.
  \end{equation}
  We consider the class of
  Csisz\`ar density functions
 \begin{subequations}
  \begin{equation}
   \label{eq:5}
   F:[0,\infty)\to [0,+\infty]\quad\text{l.s.c.~and convex},\quad
   F(1)=0,
 \end{equation}
 with recession constant defined by
  \begin{displaymath}
   F'(\infty):=\lim_{r\to\infty}\frac{F(r)}r=\sup_{r>0}\frac{F(r)}{r-1},
 \end{displaymath}
  and the corresponding class of homogeneous perspective functions
 \begin{equation}
   \label{eq:3}
   \begin{gathered}
     H:[0,\infty)\times [0,\infty)\to [0,+\infty]\quad \text{l.s.c.,
       convex, and positively $1$-homogeneous},\\
     H(\theta r,\theta s)=\theta H(r,s),\quad
     H(r,r)=0\quad\text{for every }r,s,\theta\ge0.
   \end{gathered}
 \end{equation}
 There is a one-to-one correspondence between the two classes given by
 the formula
 \begin{equation}
   \label{eq:4}
   F(r)=H(r,1),\quad
   H(r,s)=
   \begin{cases}
     sF(r/s)&\text{if }s>0,\\
     F'(\infty)&\text{if }s=0.
   \end{cases}
 \end{equation}
 \end{subequations}
 
 \begin{definition}
   \label{def:Csiszar}
   Let $F,H$ be as in {\upshape (\ref{eq:5},b)} and
   let $\mu_0,\mu_1\in \calM(\Omega)$ with Lebesgue decomposition
   $\mu_0=\varrho\mu_1+\mu_0^\bot$ as in \eqref{eq:Lebesgue}.
   The \emph{Csiz\'ar divergence} associated with $F$
   is 
   defined as
   \begin{equation}
     \mathscr{F}(\mu_0\mid \mu_1):=
     \int_\Omega F(\varrho)\,\dd\mu_1 +F'(\infty)\mu_0^{\bot}(\Omega).
     \label{eq:9}
   \end{equation}
   The $\mathscr H$-perspective functional is defined as 
   \begin{equation}
     \label{eq:8}
     \mathscr H(\mu_0\mid\mu_1):=
     \int_\Omega H(\varrho_0,\varrho_1)\,\dd\lambda
   \end{equation}
   where $\mu_i=\varrho_i\lambda\ll\lambda$, $i=0,1$, and $\lambda$ is
   any common dominating measure.\\
   If $F$ and $H$ are related by \eqref{eq:4} then
   \begin{equation}
     \label{eq:11}
     \mathscr F(\mu_0\mid\mu_1)=\mathscr H(\mu_0\mid \mu_1)\quad
     \text{for every }\mu_0,\mu_1\in \calM(\Omega).
   \end{equation}
 \end{definition}
Notice that \eqref{eq:8} does not depend on the choice of the dominating
 measure $\lambda$, since the function $H$ 
 is positively $1$-homogeneous.
 
\eqref{eq:11} can be easily checked by observing that
$\lambda:=\mu_1+\mu_0^\bot$ is a dominating measure for the couple
$\mu_0,\mu_1$; if $B_0,B_1$ are measurable subsets of
$\Omega$ such that
\begin{displaymath}
  B_0\cap B_1=\emptyset,\quad
  \Omega=B_0\cup B_1,\quad
  \mu_0^\bot(B_1)=0,\quad
  \mu_1(B_0)=0 ,
\end{displaymath}
we can easily calculate the densities $\varrho_0,\varrho_1$ by
\begin{displaymath}
  \varrho_0(x):=
  \begin{cases}
    1&\text{if }x\in B_0\\
    \varrho(x)&\text{if }x\in B_1    
  \end{cases},
  \qquad
  \varrho_1(x):=
  \begin{cases}
    0&\text{if }x\in B_0\\
 1&\text{if }x\in B_1
\end{cases}
\end{displaymath}
so that
\begin{align*}
  \int_\Omega &H(\varrho_0,\varrho_1)\,\dd\lambda
  =
  \int_{B_0} H(\varrho_0,\varrho_1)\,\dd\lambda+
    \int_{B_1} H(\varrho_0,\varrho_1)\,\dd\lambda
    \\&=\int_{B_0} H(1,0)\,\dd\mu_0^\bot+
  \int_{B_1} H(\varrho,1)\,\dd\mu_1=
  F'(\infty)\mu_0^\bot(B_0)+
  \int_{B_1} F(\varrho)\,\dd\mu_1
  \\&=F'(\infty)\mu_0^\bot(\Omega)+
  \int_{\Omega} F(\varrho)\,\dd\mu_1=\mathscr F(\mu_0\mid\mu_1).
\end{align*}
\noindent An important class of entropy functions is provided by the power like functions which have the following explicit formulas \begin{align*}
E_p(s):=\begin{cases}\frac{1}{p(p-1)}(r^p-p(r-1)-1)\qquad &\textnormal{if } p\neq 0,1\\
  r\log r-r+1 &\textnormal{if  }p=1\\
		r-1-\log r &\textnormal{if }p=0.
		\end{cases}
                                                                                                                                     \end{align*}


\noindent For $p=1$, the entropy function $E_1(r)=r\log r-r+1$ generates
the well known Kullback-Leibler divergence, 
often referred to as \textit{relative logarithmic entropy}.
Notice that $E_1$ is superlinear, so that $E_1'(\infty)=+\infty$
and its corresponding perspective function is
\begin{equation}
  \label{eq:71}
  H_{\KL}(r_0,r_1):=
  \begin{cases}
    r_0(\ln r_0-\ln r_1)+r_1-r_0&\text{if $r_0,r_1>0$,}\\
    r_1&\text{if $r_0=0$}\\
    +\infty&\text{if $r_0>0, r_1=0$.}
  \end{cases}
\end{equation}
\begin{definition}[Kullback-Leibler divergence (relative logarithmic entropy)]
  Let $\mu_0$ and $\mu_1$ be two finite nonnegative measures.
  The logarithmic entropy of $\mu_0$ with respect to $\mu_1$ is given
  by
  the Csisz\`ar functional associated to $E_1(r):=r\log r-(r-1)$:
  \begin{subequations}
    \begin{align}\label{relative entropy}
      \epy(\mu_0\mid\mu_1)&=\begin{cases}
        \displaystyle
        \int_{\Omega}\big(\varrho\log \varrho-\varrho+1\big)\,\dd\mu_1\qquad &\text{if }\mu_0=\varrho\mu_1\\
        +\infty&\text{otherwise}.
      \end{cases}\\
      \label{eq:72}
                          &=\int_{\Omega} H_{\KL}(\varrho_0,\varrho_1)\,\dd \mu,\quad
                            \mu_i=\varrho_i\mu\ll\mu,
    \end{align}
    \end{subequations}
  \end{definition}
The functionals $\mathscr F,\mathscr H$ admit a useful dual
representation.
Let us denote by $\mathrm B_b(\Omega)$ the set of bounded Borel functions on
$\Omega$ and by
$F^*:\mathbb{R}\rightarrow (-\infty,+\infty]$
the Legendre conjugate function of $F$, given by \[ 
F^*(\phi)=\sup_{s\geq 0}\big(s\phi -F(s) \big).
\]
We introduce the closed convex subsets $\mathfrak F,\mathfrak{H}$ of $\mathbb{R}^2$
given by
\begin{align*}
    \mathfrak{F}:={}&\{(\phi,\psi)\in\mathbb{R}^2:\,\,\psi\leq -F^*(\phi)
    \}=\{(\phi,\psi)\in\mathbb{R}^2:\,\,r\phi+\psi\leq
                      F(r)\,\,\,\forall r>0 \}\\
  \mathfrak{H}:={}&
    \{(\phi,\psi)\in\mathbb{R}^2:\,\,r\phi+s\psi\leq H(r,s)\ \ \forall\, r,s>0
    \}.
\end{align*}
Since $F$
is lower semicontinuous, it can be recovered from $F^*$ and $\mathfrak{F}$ by the Fenchel-Moreau formula \cite{LMS18} \[ 
F(r)=\sup_{\phi\in \mathbb{R}}(r\phi-F^*(\phi))=\sup_{(\phi,\psi)\in\mathfrak{F}} r\phi+\psi.
\]
Similarly, we have
\begin{displaymath}
  H(r,s)=\sup_{(\phi,\psi)\in\mathfrak{H}} r\phi+s\psi,
\end{displaymath}
and
$\mathfrak F=\mathfrak H$ if \eqref{eq:4} holds.
\begin{theorem}\label{dualformula_entropy}
  For every $\mu_0,\mu_1\in\mathcal{M}(\Omega)$ we have
  \begin{align*}
    \mathscr{F}(\mu_0\mid \mu_1)&=\sup\Big\{
    \int_\Omega\phi\,\dd\mu_0+\int_\Omega\psi\,\dd\mu_1:\phi,\psi\in
                                  \mathrm B_b(\Omega), \ (\phi(x),\psi(x))\in \mathfrak{F}
                                  \,\,\,\forall x\in \Omega\Big\},\\
    \mathscr{H}(\mu_0\mid \mu_1)&=\sup\Big\{
    \int_\Omega\phi\,\dd\mu_0+\int_\Omega\psi\,\dd\mu_1:\phi,\psi\in
    \mathrm B_b(\Omega), \ (\phi(x),\psi(x))\in \mathfrak{H}
    \,\,\,\forall x\in \Omega\Big\}.
  \end{align*}
\end{theorem}
\begin{proof}
\cite[Th. 2.7]{LMS18}
\end{proof}

\nc

\subsection{Hellinger distances}
We
consider a specific
example of perspective functionals $\mathscr H$,
which gives raise to the Hellinger distances.
\begin{definition}
For $\mu_0,\mu_1\in\calM(\Omega)$ and $p\in [1,+\infty)$ the
$p$-Hellinger distance is defined by
\begin{equation}
  \label{eq:2}
  \hed^p_p(\mu_0,\mu_1):=\Vert\varrho_0^{1/p}-\varrho_1^{1/p}\Vert^p_{L^p(\Omega,\lambda)}
  =\int_\Omega\Big|\varrho_0^{1/p}-\varrho_1^{1/p}\Big|^p\,\dd \lambda
\end{equation}
where $\mu_i=\varrho_i\lambda\ll\lambda$, $i=0,1$, and $\lambda$ is an
arbitrary dominating
measure.
\end{definition}
 Notice that
the above definition corresponds to \eqref{eq:8}, \eqref{eq:9} for the
choices
\begin{equation}
  H_p(r,s):=\left|r^{1/p}-s^{1/p}\right|^p,\quad
  F_p(r)=\left|r^{1/p}-1\right|^p.\label{eq:12}
\end{equation}
An immediate consequence of the above definition, choosing
$\lambda=\mu_0+\mu_1$  is the 
uniform bound
\begin{displaymath}
  \hed_p^p(\mu_0,\mu_1)\le \mu_0(X)+\mu_1(X).
\end{displaymath}
For $p=1$ the definition above gives the usual
total variation distance, {which we will still denote by $\hed_1$}. The total variation distance and the $L^p$-Hellinger distance $\hed_p$ induce the same topology on the space $\calM(\Omega)$ and the following relation holds. 



\nc
\begin{theorem}
 
  Let $q\in (1,\infty]$ be the conjugate exponent of $p$.
  For every $p>1$ \nc
  and arbitrary nonnegative finite measures $\mu_0$ and $\mu_1$ in $\calM(\Omega)$, 
\begin{equation}\label{hell_tot_var}
  \hed_p^p(\mu_0,\mu_1)\leq \hed_1(\mu_0,\mu_1) \leq
   c_p\nc \big( \mu_0(\Omega)^{1/q}+\mu_1(\Omega)^{1/q}\big) \hed_p(\mu_0,\mu_1),
\end{equation}
 where $c_p:=\max (p/2,1)$.\nc
\end{theorem}
\begin{proof}
  {The first part  of \eqref{hell_tot_var} follows immediately by the
  representation \eqref{eq:2} and the elementary inequality}
  \[ 
    \left|a^{1/p}-b^{1/p}\right|^p\leq \abs{a-b}\quad
    \text{for every }a,\,b\geq 0.
\]

\noindent The second inequality of \eqref{hell_tot_var} is a
consequence of
\begin{equation}
  \left|a^p-b^p\right|\leq c_p\,\abs{a-b}(a^{p-1}+b^{p-1}),\quad
  a,b\geq 0,
\label{eq:10}
\end{equation}
 which can be easily obtained by integration (without loss of
generality we can assume $a\le b$)
\begin{displaymath}
  b^p-a^p=p(b-a)\int_0^1 ((1-t) a+tb)^{p-1}\,\dd t=p(b-a)I
\end{displaymath}
where
\begin{displaymath}
  I=\int_0^1 ((1-t) a+tb)^{p-1}\,\dd t\le
  \begin{cases}
    \int_0^1(1-t)a^{p-1}+tb^{p-1}\,\dd t=\frac
    12(a^{p-1}+b^{p-1})&\text{if }p\ge 2,\\
    \int_0^1 \Big( (1-t)^{p-1}a^{p-1}+t^{p-1}b^{p-1}\Big)\,\dd t\le \frac
    1{p}(a^{p-1}+b^{p-1})
    &\text{if }p\le 2.
  \end{cases}
\end{displaymath}
\eqref{eq:10}
with the choices $a=\varrho_0^{1/p}$ and $b=\varrho_1^{1/p}$, combined with H\"older
inequality, yields
\begin{align*}
  \hed_1(\mu_0,\mu_1)&=\int_{\Omega} \abs{\varrho_0-\varrho_1}
                       \,\dd\lambda\leq
                       c_p\int_{\Omega}\abs{\big( \varrho_0^{1/p}-\varrho_1^{1/p}\big) \big(\varrho_0^{1/q}+\varrho_1^{1/q} \big)}\,\dd\lambda\\
                     &\leq c_p\|
                       \varrho_0^{1/p}-\varrho_1^{1/p}\|_{L^p(\Omega,\lambda)}
                       \|\varrho_0^{1/q}+\varrho_1^{1/q}\|_{L^q(\Omega,\lambda)}
                       \\&\leq c_p\|
                       \varrho_0^{1/p}-\varrho_1^{1/p}\|_{L^p(\Omega,\lambda)}
                        \big( \|\varrho_0^{1/q}\|_{L^q(\Omega,\lambda)}+\|\varrho_1^{1/q}\|_{L^q(\Omega,\lambda)}\big)
  \\&= c_p \hed_p(\mu_0,\mu_1)(\mu_0(\Omega)^{1/q}+\mu_1(\Omega)^{1/q}).
\end{align*}
\end{proof}
%
An interesting characterization of $\hed_2$ in terms of $\KL$ is
provided by the following property \cite{LMS18}:
\begin{proposition}
  For any two  measures $\mu_0$ and $\mu_1$ in $\calM(\Omega)$
  \begin{equation}
    \label{eq:70}
    \hed_2^2(\mu_0,\mu_1)=
    \min_{\mu\in \calM(\Omega)}\KL(\mu,\mu_0)+\KL(\mu,\mu_1).
  \end{equation}
  In particular
  \begin{equation}
   \hed_2^2(\mu_0,\mu_1)\leq \epy(\mu_0\mid \mu_1).\label{eq:83}
 \end{equation}
\end{proposition}
\begin{proof}
  Recalling \eqref{eq:71} and \eqref{eq:72}, \eqref{eq:70} follows by the simple calculation
  \begin{displaymath}
    \min_{r\ge 0} H_\KL(r,r_0)+H_\KL(r,r_1)
    =r_0+r_1-2\sqrt{r_0r_1}=H_2^2(r_0,r_1),
  \end{displaymath}
  attained at $r=\sqrt{r_0r_1}$.
 \end{proof}
 We now look at the Hellinger distance in its dual formulation. We
 focus on a `static-dual' formulation first and then we proceed to the
 dynamic dual formulation in terms of subsolution of the equation
 $\partial \zeta_s+(p-1)\zeta_s^q=0$. This expression will play a crucial
 role in the contraction result of Proposition \ref{contr_Hell_thm} and the
 regularizing estimates of Theorems \ref{reg_1_th} and \ref{cor_stima}.
 In the next computation we adopt the convention to write
 \begin{displaymath}
   x^a:=
   \begin{cases}
     x^a&\text{if }x>0,\\
     0&\text{if }x=0,\\
     -(-x)^a&\text{if }x<0,
   \end{cases}
   \qquad\text{for every }x\in \erre,\ a>0.
 \end{displaymath}
 \noindent \begin{corollary}
   \label{cor:staticdual}
Let $p\in (1,\infty)$ and $q$ be the conjugate of $p$.
The Hellinger distance admits the following dual formulation: \begin{equation}\label{formulation_sup} \begin{split}
\hed_p^p(\mu_0,\mu_1)=\sup
\Big\{&\int_{\Omega}\psi_1\,\dd\mu_1+\int_{\Omega}\psi_0\,\dd\mu_0:
\psi_0,\psi_1\in \mathrm B_b(\Omega)
\\&
\psi_0,\psi_{1}< 1,\ 
(1-\psi_0^{q-1})(1-\psi_1^{q-1})\geq 1\Big\}.\end{split}
\end{equation}
\end{corollary}
 \begin{proof}
 The result is a consequence of Theorem \ref{dualformula_entropy} and
the
computation of the convex set $\mathfrak F_p$
associated to the perspective function $F_p$ of \eqref{eq:12};
it is sufficient to prove that
\begin{equation}
  \mathfrak{F_p}=\{(\psi_0,\psi_1)\in \erre^2: \psi_i< 1\,\,i=0,1,\
  (1-\psi_0^{q-1})(1-\psi_1^{q-1})\geq 1 \}.\label{eq:14}
\end{equation}
In order to show \eqref{eq:14} we first compute the Legendre transform
of $F_p$, obtaining
\begin{displaymath}
  F_p^*(\psi)=\sup_{r>0} r\psi-|r^{1/p}-1|^p=
  \sup_{s>0} s^p\psi-|s-1|^p
  =
  \begin{cases}
    \displaystyle \frac\psi{(1-\psi^{q-1})^{p-1}}&\text{if }\psi<1,\\
    +\infty&\text{if }\psi\ge 1.
  \end{cases}
\end{displaymath}
Recalling that $(q-1)(p-1)=1$, the inequality $-\psi_1\ge F_p^*(\psi_0)$ for $\psi_0,\psi_1\in
\erre$ is equivalent to 
\begin{displaymath}
  \psi_0<1\quad\text{and}\quad
  -\psi_1^{q-1}(1-\psi_0^{q-1}) \ge
  \psi_o^{q-1}=1-(1-\psi_0^{q-1}).
\end{displaymath}
We then obtain
\begin{displaymath}
  \begin{aligned}
    (\psi_0,\psi_1)\in \mathfrak F_p\quad &\Leftrightarrow\quad \psi_1\le -
    F_p^*(\psi_0)\\& \Leftrightarrow\quad \psi_0<1,\ \psi_1<1,\
    (1-\psi_0^{q-1})(1-\psi_1^{q-1})\ge 1,
  \end{aligned}
\end{displaymath}
which yields \eqref{eq:14}.
 \end{proof}
\noindent The dynamic counterpart of the dual formulation is outlined in the proposition below.
\begin{proposition}\label{prop:helldyndual}
Let $p\in (1,+\infty)$ and let $q$ be the conjugate of $p$. For every
$\mu_0$, $\mu_1$ in $\calM(\Omega),$
\begin{equation}\label{hell_dyn_dual}
  \begin{aligned}
    \hed_p^{p}(\mu_0,\mu_1)=\sup
    \Big\{&\int_{\Omega}\zeta_1\,\dd\mu_1-\int_{\Omega}\zeta_0\,\dd\mu_0:
    \\&\text{
    } \zeta\in\textnormal{C}^1([0,1],\mathrm B_b(\Omega)),\,\,\,
    \partial_t\zeta_t+(p-1)\zeta_t^q \leq 0\Big\}.
  \end{aligned}
\end{equation} 
\end{proposition}
\begin{proof}
  First of all we manipulate the formulation \eqref{hell_dyn_dual} so
  that we can maximize with respect to one function only.
  We first observe that replacing, e.g.~$\psi_i$ by $\psi_{i,\eps}:=\psi_i-\eps$,
  $\eps>0$, the couple $(\psi_{0,\eps},\psi_{1,\eps})$ is still admissible
  and
  \begin{displaymath}
    \sum_i\int_\Omega \psi_i\,\dd\mu_i=\lim_{\eps\downarrow0}
    \sum_i\int_\Omega \psi_{i,\eps}\,\dd\mu_i,
  \end{displaymath}
  so that it is not restrictive to assume $\sup \psi_i<1$ in
  \eqref{formulation_sup}.
  
  that for every choice of $\psi_0\in \mathrm B_b(\Omega)$ satisfying
  $\sup\psi_0<1$ the best selection of $\psi_1$ in order to maximize
  $\sum_i \int_\Omega \psi_i\,\dd\mu_i$ is given by
  \begin{displaymath}
    \psi_1=-F^*_p(\psi_0)=
    \frac{-\psi_0}{\big(1-\psi_0^{q-1}\big)^{p-1}}.
  \end{displaymath}
  Setting $\zeta_0:=-\psi_0$ we obtain the formula
  \begin{displaymath}
    \hed_p^p(\mu_0,\mu_1)=
    \sup_{\zeta_0\in \mathrm B_b(\Omega),\ \zeta_0>-1}
    \left(\int_\Omega\frac{\zeta_0}{\big(1+\zeta_0^{q-1}\big)^{p-1}}\,\dd\mu_1-\int_\Omega\zeta_0\,\dd\mu_0\right)
    .
\end{displaymath}
%
%
On the other hand
we observe that the function
$\zeta_1:=\frac{\zeta_0}{\big(1+\zeta_0^{q-1}\big)^{p-1}}$
corresponds to the solution at time $t=1$ of
\begin{equation}
\left\{
  \begin{aligned}
    \partial_t \zeta(t,x)+(p-1)\zeta^q(t,x)&=0&&\text{in }[0,1]\times \Omega,\\
    \zeta(0,x)&=\zeta_0(x)&&\text{in }\Omega.
  \end{aligned}
  \right.
\label{eq:20}
\end{equation}
and by the comparison theorem for ordinary differential equation, any
subsolution to \eqref{eq:20} will satisfy $\zeta(1,x)\le \zeta_1(x)$.
\end{proof}
\subsection{Kantorovich-Wasserstein and Hellinger-Kantorovich distances}
\subsubsection*{Kantorovich-Wasserstein distance.} The standard
definition of the Kantorovich\phantom{{-}}\-Wasserstein distance arises in a natural
way in the frame of optimal transport. Here we  recall the definition
only and we refer to \cite{AGS08,Villani09} for further details.

We will deal with a complete and separable metric space
$(X,\mathsf{d})$; we denote by
$\mathscr{B}(X)$ its Borel $\sigma$-algebra and by $\calP(X)$ the space of Borel probability
measures on $X$. For $p\ge 1$ we set
\begin{displaymath}
  \mathcal P_p(X):=\Big\{\mu\in \mathcal P(X):
  \int_X \mathsf d^p(x,o)\,\dd\mu(x)<+\infty\Big\},
\end{displaymath}
where $o$ is an arbitrary point of $X$ (the definition is independent
of the choice of $o$).

If $\boldsymbol t:X\to Y$ is a Borel map between two metric spaces,
we denote by $\boldsymbol t_\sharp:\mathcal P(X)\to\mathcal P(Y)$ the
corresponding push-forward operation, defined by
\begin{displaymath}
  \boldsymbol t_\sharp\mu(B):=\mu(\boldsymbol t^{-1}(B))\quad
  \text{for every }B\in \mathscr B(Y).
\end{displaymath}
In particular, when we consider the canonical cartesian projections
$\pi^i:X\times X\to X$ defined by $\pi^i(x_0,x_1):=x_i$, $i=0,1$,
and a general measure (also called transport plan) $\boldsymbol\mu\in
\mathcal P(X\times X)$,
the measures $\mu_i=\pi^i_\sharp \mu$ are the marginals of
$\boldsymbol \mu$.
\begin{definition}\label{OTp}
Let $p\in [1,\infty)$. For any $\mu_0,\,\mu_1\in\calP_p(X)$ the
$p$-Kantorovich-Wasserstein distance is defined by
\begin{displaymath}
  \begin{aligned}
    \sfW_p^p(\mu_0,\mu_1):=\min\Big\{&\int
    \mathsf{d}^{p}(x_0,x_1)\,\dd \boldsymbol{\mu}(x_0,x_1):
    \boldsymbol{\mu}\in\mathcal P(X\times X),\
    \pi^i_\sharp\boldsymbol\mu=\mu_i\,\,, i=0,1\Big\}.
  \end{aligned}
\end{displaymath}
\end{definition}
\nc
As we will see, a  key ingredient we will extensively use in our
arguments
is given by the dynamic dual formulation of the Wasserstein distance,
in terms of the subsolutions of the Hamilton-Jacobi equation.
Such a result, which has been formulated in different form
by \cite{Otto-Villani00,AGS14I,AMS15preprint,AES16}, 
holds if $(X,\sfd)$ is a \emph{length space},
i.e.~if for every $x_0,x_1\in X$ and every $\theta>1/2$ there exists
an approximate mid-point $x_{\theta}\in X$ such that
\begin{displaymath}
  \max\big(\sfd(x_{0},x_{\theta}),\sfd(x_\theta,x_1)\big)\le \theta\sfd(x_0,x_1).
\end{displaymath}
We denote by $\tnLip_b(X)$ the Banach space of bounded Lipschitz
functions $f:X\to \erre$ endowed with the norm
\begin{displaymath}
  \|f\|_{\tnLip_b}:=\sup_{x\in X} |f|+\tnLip(f,X),\quad
  \tnLip(f,X):=\sup_{x,y\in X,\ x\neq y}\frac{|f(x)-f(y)|}{\sfd(x,y)}.
\end{displaymath}
\begin{proposition}
  \label{prop:dual_dyn}
  If $(X,\sfd)$ is a length space then for every $\mu_0,\mu_1\in \calP_p(X)$
  \begin{equation}\label{dyn_Wass}\begin{split}
\frac{1}{p}\sfW_p^p(\mu_0,\mu_1)=\sup\Big\{&\int_X\zeta_1\,\dd\mu_1-\int_X\zeta_0\,\dd\mu_0:\\
&\zeta\in \rmC^1([0,1],\tnLip_b(X)) \text{ s.t. }\partial_t\zeta_t+\frac{1}{q}\abs{\tD \zeta_t}^q\leq 0 \Big\},
\end{split}
\end{equation}
where $q$ is the conjugate of $p$.
\end{proposition}
\begin{proof}
  Let $\mu_0,\mu_1\in \calP_p(X)$; since $(X,\sfd)$ is a length space,
  also $(\calP_p(X),\sfW_p)$ is a length space, so that for every
  $a>1$ we can find a Lipschitz curve $\mu:[0,1]\to\calP_p(X)$ such that
  \begin{equation}\label{eq:75}
    |\dot\mu_t|_{\sfW_p}:=\limsup_{h\to0}\frac{\sfW_p(\mu_t,\mu_{t+h})}{|h|}\le
    a\sfW_p(\mu_0,\mu_1)
    \quad\text{for every }t\in [0,1].
  \end{equation}
  It follows that for every curve $\zeta\in \rmC^1([0,1],\tnLip_b(X))$
  the map $t\mapsto\int_X \zeta_t\,\dd\mu_t$ is
  Lipschitz continuous and by \cite[Lemma 6.4, Theorem 6.6]{AMS15preprint}
  \begin{displaymath}
    \int_X \zeta_1\,\dd\mu_1-
    \int_X\zeta_2\,\dd\mu_2\le
    \int_0^1\int_X \Big(\partial_t\zeta_t+\frac 1q|\mathrm D
    \zeta_t|^q(X)\Big)\,\dd\mu_t\,\dd t+
    \frac 1p\int_0^1 |\dot\mu_t|_{\sfW_p}^p\,\dd t;
  \end{displaymath}
  if $\zeta$ is also a subsolution to the Hamilton-Jacobi equation
  \begin{equation}
    \label{eq:74}
    \partial_t \zeta+\frac 1q|\mathrm D\zeta|^q\le 0\quad
    \text{ in }[0,1]\times X,
  \end{equation}
  then the previous inequality, the bound \eqref{eq:75} on the metric velocity
  $|\dot\mu_t|_{\sfW_p}$ and the arbitrariness of $a>1$ yield
  \begin{displaymath}
     \int_X \zeta_1\,\dd\mu_1-
     \int_X\zeta_2\,\dd\mu_2\le\frac 1p \sfW_p^p(\mu_0,\mu_1)
     \quad\text{for every }\zeta\in \rmC^1([0,1];\tnLip_b(X)) \text{
       as in }\eqref{eq:74}.
  \end{displaymath}
  On the other hand, for every $a<1$ we can use the Hopf-Lax semigroup
  \begin{displaymath}
    \mathsf Q_t\zeta(x):=\inf_{y\in X}\frac 1{qt^{q-1}}\sfd^q(x,y)+\zeta(y)
  \end{displaymath}
  and Kantorovich duality for the Wasserstein distance
  to find $\zeta_0\in \tnLip_b(X)$ such that
  \begin{displaymath}
    \int_X \mathsf Q_1\zeta_0\,\dd\mu_1-
    \int_X \zeta_0\,\dd\mu_0\ge \frac ap\sfW_p^p(\mu_0,\mu_1).
  \end{displaymath}
  Using the refined estimate on the Hopf-Lax semigroup of
  \cite{AGS14I} we can show that $\zeta_t:=\mathsf Q_t\zeta_0$
  is uniformly bounded in $\tnLip_b(X)$, is Lipschitz continuous
  with values in $\mathrm C_b(X)$ and satisfies
  \begin{displaymath}
    \partial^+_t \zeta+\frac 1q|\mathrm D\zeta|^q\le 0\quad
    \text{ in }[0,1)\times X,\quad
    \partial^+_t \zeta_t(x)=
    \lim_{h\downarrow0}h^{-1}\big(\zeta_{t+h}(x)-\zeta_t(x)\big).
  \end{displaymath}
  By using a rescaling argument of \cite{AES16} and the smoothing
  technique
  of the proof of \cite[Theorem 8.12]{LMS18} we conclude.  
\end{proof}
\subsubsection*{Hellinger-Kantorovich distance.}
After Hellinger-Kakutani and Kantorovich-Wasserstein distances,
we recall the definition of a third distance between probability
measures, 
that plays a role in the main contributions of this work.

Let
$(X,\mathsf{d})$ be a separable complete metric space. The
Hellinger-Kantorovich distances are defined on the space of finite
nonnegative Borel measures $\calM(X)$ and they do not require
measures to have the same mass. As in the previous cases of $\hed_p$
or $\sfW_p$, the
Hellinger-Kantorovich distances admit different formulations that we
summarize below. 
Here we focus on the family of distances $\HK_\alpha$
depending on a tuning parameter $\alpha>0$;
they correspond to the case $\HK_{\alpha,\beta}$ of \cite{LMS16}
with the choice $\beta:=4$.
In the even more specific case $\alpha=1$,
$\HK_1$ coincides with the distance $\HK$ which has been extensively studied
in \cite{LMS18}. The general case $\alpha\neq 1$ can be reduced
to the case $\alpha=1$ by rescaling the distance $\sfd$ by a factor
$\alpha^{-1/2}$.

The first formulation comes from the
Logarithmic-Entropy-Transport problem, where the constraints on the
marginals typical of optimal transport problems \eqref{OTp} are
relaxed by the introduction of two penalizing functionals. The primal
formulation of the Hellinger-Kantorovich distance is the following:
\begin{definition}
  For any $\mu_0\,\mu_1\in\calM(X)$,
  \begin{displaymath}
    \begin{split}
      \HK_\alpha^2(\mu_0,\mu_1):=\min\Big\{&\sum_{i}\KL(\gamma_i|\mu_i)+\int
\ell_\alpha(\mathsf d(x_0,x_1))\,d\boldsymbol{\gamma}:\\
&\text{ }\boldsymbol{\gamma}\in\calM(X\times X),
\text{ }\pi^i_\sharp\boldsymbol{\gamma}=\gamma_i\ll\mu_i , \,\,\,\,i=0,1\Big\},
	\end{split}
\end{displaymath}
where $\ell_\alpha:[0,+\infty)\to [0,+\infty]$ is the cost function
defined by \eqref{eq:76}.
\end{definition}
A direct comparison with \eqref{eq:70} by restricting $\ggamma$
to plans of the form 
$\ggamma:=\iota_\sharp\mu$ where $\mu\in \calM(X)$ is an arbitrary
measure
dominating $\mu_i$ and $\iota:X\to X\times X$ is the diagonal identity
map
$\iota(x):=(x,x)$, immediatily yields
\begin{equation}
  \label{eq:67}
  \HK_\alpha(\mu_0,\mu_1)\le \hed_2(\mu_0,\mu_1)\quad\text{for every
    $\mu_0,\mu_1\in \calM(X)$ and $\alpha>0$.}
  \end{equation}
  \cite[Theorem 7.22]{LMS18} also shows that
  \begin{displaymath}
    \lim_{\alpha\downarrow0}\HK_\alpha(\mu_0,\mu_1)=\hed_2(\mu_0,\mu_1).
  \end{displaymath}
  \cite[Proposition 7.23, Theorem 7.24]{LMS18} provide two further
  useful bounds of $\HK_\alpha$ in terms of $\sfW_2$, when
  $\mu_0,\mu_1\in \calP_2(X)$:
  \begin{equation}
    \label{eq:79}
    \sqrt\alpha \HK_\alpha(\mu_0,\mu_1)\le \sfW_2(\mu_0,\mu_1),\quad
    \lim_{\alpha\uparrow+\infty}\sqrt\alpha \HK_\alpha(\mu_0,\mu_1)=
    \sfW_2(\mu_0,\mu_1).
  \end{equation}
\noindent The $\HK_\alpha$ distance admits an equivalent dual
formulation in terms of subsolutions to a suitable version of the
Hamilton-Jacobi equation, which can be compared with
\eqref{hell_dyn_dual} and \eqref{dyn_Wass}:
in fact, it is possible to show \cite[Section 8.4]{LMS18} that \begin{equation}\label{HK_weight}\begin{split}
\HK_{\alpha}^2(\mu_0,\mu_1)=\sup\Big\{\int_{X}\zeta_1\,\dd\mu_1-&\int_X\zeta_0\,\dd\mu_0:\text{ }\zeta\in \textnormal{C}^{1}([0,1], \tnLip_b(X))\text{ s.t. }\\
	&\partial_t\zeta_t+\frac{\alpha}{4}\abs{\tD\zeta_t}^2+\zeta_t^2\leq 0 \text{ in }[0,1]\times X\Big\}.
\end{split}
\end{equation}
\section{Metric measure spaces with curvature bounds}
\label{sec:metric-measure}
This section is dedicated to a brief review of a few notions related
to calculus and Sobolev spaces in metric measure spaces.
We refer to \cite{AGS14I} and  \cite{AGS14D} for a complete review of the topic.
\subsection{Calculus in metric measure spaces: basic notions }

%

 \newcommand{\www}{{w}}
\noindent  Let $(X,\mathsf{d})$ be a complete and
separable metric space, endowed with a Borel positive measure
$\mathfrak{m}$ satisfying the growth condition \eqref{eq:80}
and $\operatorname{supp}(\gm)=X$.
As we already mentioned in the Introduction,
on this class of metric measure space
it is possible to introduce an effective metric counterpart
of the classic Dirichlet energy form in Euclidean spaces and of the
corresponding Sobolev spaces.
In the following, we will recall the basic notions only, which are
strictly necessary to understand the main results of the work, by
adopting the Cheeger point of view.
\begin{definition}
A function $G\in L^2(X,\mathfrak{m})$ is a relaxed gradient of $f\in L^2(X,\mathfrak{m})$ if there exist Borel $\mathsf{d}$-Lipschitz functions $f_n\in L^2(X,\mathfrak{m})$ such that: \begin{itemize}
\item[a)] $f_n\to f$ in $L^2(X,\mathfrak{m})$ and $\abs{\textnormal{D}f_n}$ weakly converge to $\tilde{G}$ in $L^2(X,\mathfrak{m})$;
\item[b)] $\tilde{G}\leq G$ $\mathfrak{m}$-a.e. in $X$.
\end{itemize}
We say that $G$ is the minimal relaxed gradient of $f$ if its $L^2(X,\mathfrak{m})$ norm is minimal among relaxed gradients. We shall denote by $\abs{\textnormal{D}f}_\www$ the minimal relaxed gradient. 
\end{definition}
\noindent The minimal relaxed gradient is used to
give an integral formulation of the Cheeger energy
\eqref{eq:54},
which can be represented as
\begin{displaymath}
\mathsf{Ch}(f)= \frac{1}{2}\int_X
\abs{\text{D}f}^2_\www\,\dd\mathfrak{m}
\quad\text{if $f$ has a $L^2$ relaxed gradient},
\end{displaymath}
and set equal to $+\infty$ if $f$ has no relaxed gradients.
The Cheeger energy is a convex, $2$-homogeneous lower semicontinuous functional on $L^2(X,\mathfrak{m})$ with dense domain $\mathcal{D}(\mathsf{Ch})$ \cite[Th. 4.5]{AGS14I}.
From the lower semicontinuity of $\mathsf{Ch}$ it follows that the domain $\mathcal{D}(\mathsf{Ch})$ endowed with the norm \[
  \norm{f}_{\rmW^{1,2}}:=\sqrt{\norm{f}_2^2+\norm{\abs{\text{D}f}_\www}_{2}^2}
 \]
is a Banach space, which is called $\rmW^{1,2}(X,\mathsf{d},\mathfrak{m})$. In general it is not a Hilbert space and this causes the potential non linearity of the heat flow. 
The following proposition summaries some useful properties of the minimal relaxed gradient, which will be helpful for our purposes.

\begin{proposition}\label{calculus}
Let $f\in L^2(X,\mathfrak{m})$. Then the following properties hold: \begin{itemize}
\item[a)] $\abs{\tD f}_\www=\abs{\tD g}_\www$ $\mathfrak{m}$-a.e. on $\{f-g=c\}$ for all constants $c\in \erre$ and $g\in L^2(X,\mathfrak{m})$ with $\mathsf{Ch}(g)<+\infty$;

\item[b)] $\phi(f)\in \mathcal{D}(\mathsf{Ch})$ and $\abs{\tD
    \phi(f)}_\www
  \leq \abs{\phi'(f)} \abs{\tD f}_\www$ for any Lipschitz function $\phi$ on an interval $I$ containing the image of $f$; the inequality refines to the equality $\abs{\tD\phi(f)}_\www=\abs{\phi'(f)} \abs{\tD f}_\www$ if in addition $\phi$ is nondecreasing;
\item[c)]  if $f,g\in \mathcal{D}(\mathsf{Ch})$ and $\phi:\erre\to \erre$ is a nondecreasing contraction, then \[ 
\abs{\tD(f+\phi(g-f))}_\www^2+\abs{\tD(g-\phi(g-f))}_\www^2\leq \abs{\tD f}_\www^2+\abs{\tD g}_\www^2\qquad\mathfrak{m}\text{-a.e. in }X.
\]
\end{itemize}
\end{proposition}
\begin{proof}
\cite[Prop. 4.8]{AGS14I}
\end{proof}
\subsection{Gradient flow of the Cheeger energy in metric-measure spaces}
The metric-measure counterpart of the Laplacian operator
can be defined in terms of the element of minimal $L^2$-norm in
the subdifferential $\partial\mathsf{Ch}$ of $\mathsf{Ch}$.
$\partial\mathsf{Ch}$ is the multivalued operator in $L^{2}(X,\mathfrak{m})$ defined for all $f\in\mathcal{D}(\mathsf{Ch})$ by the following relation:
\[
\quad l\in\partial\mathsf{Ch}(f)\quad \Longleftrightarrow\quad \int_{X}l(g-f)\,\dd\mathfrak{m}\leq \mathsf{Ch}(g)-\mathsf{Ch}(f)\quad \forall\,g\in L^{2}(X,\mathfrak{m}).\]

\begin{definition}[Metric-measure Laplacian]
\textit{The metric-measure Laplacian $\Delta f$ of $f\in L^2(X,\mathfrak{m})$ is defined for any $f$ such that $\partial\mathsf{Ch}(f)\neq \emptyset$. For those $f$, $-\Delta f$ is the element of minimal $L^2(X,\mathfrak{m})$ norm in $\partial\mathsf{Ch}(f)$. }
\end{definition}
\noindent The domain of $\Delta$ is denoted by $\mcD(\Delta)$ and  is a dense subset of $\mcD(\mathsf{Ch})$. 
\noindent
The metric-measure heat flow can be obtained
by applying the classic theory of gradient flows in Hilbert spaces
\cite{Brezis73} and it enjoys further properties
which have been studied in \cite{AGS14I}.
More refined contraction properties will be proved
in Section \ref{sec:contraction}.
\begin{theorem}[Gradient flow of $\mathsf{Ch}$ in $L^2(X,\gm)$]
For any $f\in L^2(X,\gm)$ there exists a unique locally absolutely
continuous curve
$(0,\infty)\ni t\rightarrow \sfP_t f\in L^2(X,\gm)$ such that $\sfP_t f\rightarrow f$ in $L^2(X,\gm)$ as $t\rightarrow 0$ and \[ 
\frac{\dd}{\dd t} \sfP_t f\in-\partial\mathsf{Ch}(\sfP_t f)\qquad \text{for a.e. }t>0.
\]
The following properties hold:
\begin{enumerate}[(1)]\item 
  The curve $t\mapsto \sfP_t f$  is locally Lipschitz, $\sfP_t f\in\mcD(\Delta)$
  for any $t>0$ and it holds
  \[ \frac{\dd^{+}}{\dd t} \sfP_t f=\Delta \sfP_t f \qquad\forall t>0.
  \]
  \item The curve $t\rightarrow \mathsf{Ch}(\sfP_t f)$ is locally Lipschitz
  in $(0,+\infty)$, infinitesimal at $+\infty$ and continuous in $0$
  if $f\in \mcD(\mathsf{Ch})$.  Its right derivative is given by
  $-\Vert \Delta \sfP_t f\Vert_{L^{2}}^{2}$ for every $t>0$.
  \item The
  family of maps $(\sfP_t)_{t\ge0}$ is a strongly continuous semigroup
  of contractions in $L^2(X,\gm)$ which can be extended in a unique
  way to a strongly continuous semigroup of contractions in every
  $L^p(X,\gm)$, $1\le p<\infty$ (still denoted by $(\sfP_t)_{t\ge0}$)
  thus satisfying
  \begin{equation}
    \label{eq:41}
    \|\sfP_t f-\sfP_t g\|_{L^p(X,\gm)}\le \|f-g\|_{L^p(X,\gm)}\quad
    \text{for every }f,g\in L^p(X,\gm).
  \end{equation}
\end{enumerate}

\end{theorem}
\subsection{$\RCD(K,\infty)$ metric measure spaces}
In this subsection we briefly recall the definition and some
properties of a class of metric measure spaces which generalize the
notion of Riemannian manifolds with Ricci curvature bounded from
below. This will be the
general setting of the regularization result that we propose in
Section \ref{sec:regularization}, where, indeed, the bound on the curvature plays a direct role. \\
%
On a general metric measure space, the Cheeger energy is not a
quadratic form and this translates into a potential lack of linearity
of its  $L^2$-gradient flow  $(\mathsf{P_t})_{t\geq 0}$.
If we require the Cheeger energy to be quadratic, and hence the heat
flow to be linear, we restrict the choice of the underlying metric
domain to class of metric measure spaces which can be considered
a nonsmooth generalization of Riemannian manifolds:
among them, the so called Bakry-\'Emery curvature condition
can be used to select the class of $\RCD(K,\infty)$ metric measure
spaces
(we refer to \cite{AGS14I,AGS14D} for a complete discussion
and the other important equivalent characterization we mentioned in
the Introduction). As in the previous section,
$(X,\sfd,\gm)$ will denote a complete and separable metric measure
space satisfying the volume growth condition \eqref{eq:80}.
%
\begin{definition}[The $\RCD(K,\infty)$-condition]
  \label{def:RCD}
  $(X,\mathsf{d}, \gm)$
  is a $\RCD(K,\infty)$ metric measure space if  the Cheeger energy is
  quadratic \eqref{eq:63}, 
  every function $f\in \calD(\mathsf{Ch})\cap L^\infty(X,\gm)$
  with $|\tD f|_\www\le 1$ admits a
  $1$-Lipschitz representative (still denoted by $f$)
  and \begin{equation}\label{refinedBE}
    \abs{\tD(\mathsf{P}_t f)}_\www^2\leq \textnormal{e}^{-2Kt}\mathsf{P}_t(\abs{\tD f}_\www^2)\qquad\gm\textnormal{-a.e. in }X.
\end{equation}
\end{definition}
%
\noindent Equation \eqref{refinedBE} is
one of the equivalent formulation of the celebrated
Bakry-\'Emery condition \cite{Bakry1985}, \cite{AGS15}.
Notice that the $\RCD(K,\infty)$ condition implies in particular
that every bounded function $f\in \calD(\mathsf{Ch})$ with
$|\tD f|_\www\in L^\infty(X,\gm)$
has a Lipschitz continuous representative
(identified with $f$) satisfying
\begin{displaymath}
  \sup_X |\tD f|=\tnLip(f,X)\le \big\|\,|\tD f|_\www\,\big\|_{L^\infty(X,\gm)}.
\end{displaymath}
On $\RCD(K,\infty)$ spaces, an even stronger version of
\eqref{refinedBE} holds true,
together with crucial regularization properties which we collect in
the next statement.
\begin{theorem}
  \label{thm:cond_equivalenti}
  Let $(X,\mathsf{d},\gm)$ be a $\RCD(K,\infty)$ space.
  \begin{enumerate}[(1)]
  \item
    For every $f\in L^\infty(X,\gm)$ and $t>0$ the function
    $\sfP_t f$ has a unique continuous representative
    $\tilde\sfP_t f\in \tnLip_b(X)$
    (in the following, with a slight abuse of notation,
    we will identify $\sfP_tf$ with $\tilde\sfP_t
    f$, whenever $f\in L^\infty(X,\gm)$).
  \item
    For every $f\in \calD(\mathsf{Ch})$ with
    $f,\abs{\tD f}_\www\in L^{\infty}(X,\gm)$ and $t>0$
    \begin{equation}\label{refined_BE}
      |\tD \sfP_t f|=|\tD\sfP_t f|_\www
      \ \text{$\gm$-a.e.~in $X$,}\qquad
      \abs{\tD \mathsf{P}_t f}\leq \textnormal{e}^{-Kt}
      \mathsf{P}_t\abs{\tD f}_\www\quad\text{in $X$.}
    \end{equation}
  \item For every $f\in L^\infty(X,\gm)$ and $t>0$
    \begin{equation}\label{key_point} 
      R_K(t)
      \vert\tD \sfP_{t}f\vert^2\leq \sfP_{t}(f^2)-(\sfP_t f)^2\qquad\text{ in }X,
    \end{equation}
    where $R_k$ has been defined in \eqref{eq:66}.
    In particular
    \begin{equation}
      \label{eq:82}
      \sqrt{R_K(t)}\tnLip(\sfP_t f,X)\leq \norm{f}_{L^\infty(X,\gm)}.
    \end{equation}
  \end{enumerate}  
\end{theorem}
\begin{proof}
  Property (1) is a consequence of \cite[Corollary 4.18]{AGS15}.
  The first identity of \eqref{refined_BE} is stated in \cite[Theorem
  3.17]{AGS15}; the second one is stated in \cite[Corollary
  4.3]{Savare14}.
  \eqref{eq:72} is a consequence of the above properties and the estimate
  of \cite[Corollary 2.3(iv)]{AGS15}.
\end{proof}

 \section{Contraction properties for the Heat flow in metric measure spaces}
 \label{sec:contraction}
This section is devoted to some fairly general contraction properties
 of the heat flow in
the metric-measure setting. Our main result concerns the behaviour of
the functional
\begin{subequations}
  \label{sub:E}
  \begin{equation}
    \label{eq:35}
    \mathscr{E}(f,g):=\int_{X}E(f,g)\,\dd\gm,\quad
    f,g\in L^2(X,\gm)
  \end{equation}
  where
  \begin{equation}
    \label{eq:36}
    \text{$E:\erre^2\to \erre\cup\{+\infty\}$ is a proper, l.s.c.~and
      convex function.}
  \end{equation}
  Since $E$ is bouded from below by an affine map,
  when $\gm(X)<\infty$ the integral of \eqref{eq:35} is always well
  defined
  (possibly taking the value $+\infty$).
  In the general case, in order to avoid
  integrability issues, we will also assume that
  \begin{equation}
    \label{eq:37}
    E\text{ is nonnegative, }E(0,0)=0\quad\text{if $\gm(X)=+\infty$}.
  \end{equation}
\end{subequations}
\begin{theorem}\label{contr_convex}
  \nc
  Let $(X,\mathsf{d},\gm)$ be a metric measure space
  with the Heat semigroup
  $(\sfP_t)_{t\geq 0 }$ generated by the Cheeger
energy $\mathsf{Ch}$ in $L^2(X,\gm)$, and let $\mathscr E$ be defined as in
{\upshape (\ref{sub:E}a,b,c)}.
Then, for $f,g\in L^2(X,\gm)$
\begin{displaymath}
  \mathscr{E}(\sfP_t f,\sfP_t g)\leq \mathscr{E}(f,g)\quad
  \text{for every }t\ge0.
\end{displaymath}
\end{theorem}
\noindent We prove some useful lemmas first. The first one shows a
generalization of part $c)$
in Proposition \ref{calculus} and is the core of the proof of the main theorem.
\begin{lemma}\label{lemma_crucial}
Let $(X,\mathsf{d})$ be a metric space, let $E:\erre^2\to \erre$
be a $\textnormal{C}^2$ convex function with $1$-Lipschitz (w.r.t the
Euclidean norm) gradient
$\nabla E:\erre^2\to \erre^2$, and let $J:\erre^2\to \erre^2$ be the map
$J:=Id-\nabla E$.
For every bounded Lipschitz map
$\boldsymbol{f}:=(f_1,f_2):\erre^2\to\erre^2$,
the function $\boldsymbol{g}=(g_1,g_2):=J\circ \boldsymbol{f}$ satisfies
\begin{equation}\label{disug_diff}
  \vert\tD g_1\vert^2(x)+\vert\tD g_2\vert^2(x)\leq
  \abs{\tD f_1}^2(x)+\abs{\tD f_2}^2(x)\quad
  \text{for every }x\in X.
\end{equation}
\end{lemma}
\begin{proof}
   Since $\nabla^2E$ is positive definite and
  $J$ is $1$-Lipschitz, we observe that $A:=\mathrm D J=I-\nabla^2E$ satisfies
  \begin{equation}
    \label{eq:24}
    0\le \boldsymbol z^T A(\boldsymbol w)
    \boldsymbol z\le |\boldsymbol z|^2\quad
    \text{for every }\boldsymbol w,\, \boldsymbol z\in \erre^2.
  \end{equation}
  For every $x,y\in X$, $x\neq y$, and $f:X\to \mathbb R$ we set
  \begin{displaymath}
    R(f,x,y):=\frac{ \abs{f(x)-f(y)}}{\mathsf{d}(x,y)}\quad
    \text{so that}\quad
    |\tD f|(x)=\limsup_{y\to x}R(f,x,y).
  \end{displaymath}
  Let us now fix $x\in X$;
  it is possible to find two sequences of points $(x_i^n)_{n\in \mathbb
  N}$, $i=1,2$, such that 
\begin{equation}
\lim_{n\to+\infty} R(g_i,x,x_i^n)=\vert \tD g_i\vert(x).\label{eq:25}
\end{equation}
\nc
Taking a linear combination of the difference quotients
$R(g_i,x,x_i^n)$
with the \emph{positive} coefficients $v_i:=\vert \tD g_i\vert(x)$ it holds \[ 
\lim_{n\to +\infty}\sum_i v_i\,R(g_i,x,x_i^n)=\sum_{i}\vert \tD g_i\vert^2(x).
\]
Now, $g_i(x)=J_i(f_1(x), f_2(x))$ and hence \[ 
\lim_{n\to +\infty}v_i R(g_i,x,x_i^n)=\lim_{n\to +\infty}v_i\frac{\abs{J_i(f_1(x),f_2(x))-J_i(f_1(x_i^n),f_2(x_i^n))}}{\mathsf{d}(x,x_i^n)}.
\]
Since $J$ is $\textnormal{C}^1$,
a first order expansion at $\boldsymbol z=\boldsymbol f(x)$
with $\boldsymbol z^n_i:=\boldsymbol f(x_i^n)$ and the Lipschitz
character of $\boldsymbol f$ yield
\begin{align*}
  J(\boldsymbol z^n_i)-
  J(\boldsymbol z)&=A(\boldsymbol z)(\boldsymbol z^n_i-\boldsymbol z)+
                    o(|\boldsymbol z^n-\boldsymbol z|)
                    \\&=
  \partial_1J(\boldsymbol f(x))(f_1(x_i^n)-
   f_1(x))+\partial_2J(\boldsymbol f(x))(f_2(x_i^n)-
  f_2(x)) +o(\mathsf d(x_i^n,x)).
\end{align*}
Estimating the first component $J_1$ of $J$ along the sequence $(x^n_1)_n$
and
the second component $J_2$ of $J$ along $(x^n_2)_n$ we get for $i=1,2$
\begin{align*}
&\lim_{n\to +\infty}\frac{\abs{J_i(f_1(x),f_2(x))-J_i(f_1(x_i^n),f_2(x_i^n))}}{\mathsf{d}(x,x_i^n)}\\
  &=
    \lim_{n\to\infty}
    \Big|\partial_1J_i(f_1(x),f_2(x)) R(f_1,x,x^n_i) +
    %
    \partial_2J_i(f_1(x),f_2(x))R(f_2,x,x^n_i)\Big|\\
  &\le 
    |\partial_1J_i(f_1(x),f_2(x))|
    \limsup_{n\to\infty} R(f_1,x,x_i^n)+
    |\partial_2J_i(f_1(x),f_2(x))| \limsup_{n\to\infty} R(f_2,x,x_i^n)
   \\& \le |\partial_1J_i(f_1(x),f_2(x))| |\tD f_1|(x)+
    |\partial_2J_i(f_1(x),f_2(x))| |\tD f_2|(x)
.\end{align*}
Recalling \eqref{eq:25}, since the coefficients $v_1,v_2$ are
nonnegative,
we get
\begin{align*}
  \sum_{i}\vert \tD g_i\vert^2(x)\leq &
                                        \sum_{i,j}
                                        \hat A_{j,i}(f_1(x),f_2(x))
                                        |\tD f_j|(x) v_i,
\end{align*}
where for every $w\in \mathbb R^2$ $\hat A(\boldsymbol w)$ is the
symmetric matrix defined by 
\begin{displaymath}
  \hat A_{i,j}(\boldsymbol w):=|A_{i,j}(\boldsymbol w)|.
\end{displaymath}
\eqref{eq:24} and the next elementary Lemma yield
\begin{displaymath}
  \sum_{i}\vert \tD g_i\vert^2(x)\le \Big(\sum_i |v_i|^2\Big)^{1/2}
  \Big(\sum_i |\tD f_i|^2(x)\Big)^{1/2}
\end{displaymath}
thus obtaining \eqref{disug_diff}.
\end{proof}
\begin{lemma}
  \label{le:elementary}
  Let $A\in \mathbb R^{2\times 2}$ be a symmetric matrix
  and let $\hat A\in \mathbb R^{2\times 2}$ be defined by
  $\hat A_{ij}:=|A_{ij}|$, $i,j=1,2$.
  If 
  \begin{equation}
    \label{eq:24bis}
    0\le \boldsymbol z^T A
    \boldsymbol z\le |\boldsymbol z|^2\quad
    \text{for every }\boldsymbol z\in \erre^2,
  \end{equation}
  then also $\hat A$ satisfies
  \begin{equation}
    \label{eq:24tris}
    0\le \boldsymbol z^T \hat A
    \boldsymbol z\le |\boldsymbol z|^2\quad
    \text{and}\quad
    \boldsymbol z^T \hat A
    \boldsymbol w\le
    |\boldsymbol z|\,|\boldsymbol w|\quad
    \text{for every }\boldsymbol z,\boldsymbol w\in \erre^2.
  \end{equation}
\end{lemma}
\begin{proof}
  It is easy to check that a symmetric matrix $A$ satisfies
  $0\le \boldsymbol z^T A
  \boldsymbol z\le |\boldsymbol z|^2$ for every $\boldsymbol z\in \mathbb R^2$ 
  \eqref{eq:24bis} if and only if
  \begin{equation}
    \label{eq:27}
    0\le A_{11}\le 1,\quad 0\le A_{22}\le 1,\quad A_{12}^2\le
    A_{11}A_{22},\quad
    A_{12}^2\le
    1+A_{11}A_{22}-A_{11}-A_{22},
  \end{equation}
  and it is clear that \eqref{eq:27} is preserved if
  we replace the coefficients $A_{ij}$ by $|A_{ij}|$.
  The second inequality of \eqref{eq:24tris} follows immediately by
  the first one and 
  the Cauchy-Schwartz inequality, since
  \begin{displaymath}
    \boldsymbol z^T \hat A
    \boldsymbol w\le
    \big(\boldsymbol z^T \hat A
    \boldsymbol z\big)^{1/2}
    \big(\boldsymbol w^T \hat A
    \boldsymbol w\big)^{1/2}\le
    |\boldsymbol z|\,|\boldsymbol w|.
  \end{displaymath}
\end{proof}
\begin{lemma}\label{rmk_relax}
  Let $E:\erre^2\to \erre$
be a $\textnormal{C}^{1,1}$ convex function 
as in \eqref{eq:36} and \eqref{eq:37} with $1$-Lipschitz (w.r.t the
Euclidean norm) gradient
$\nabla E:\erre^2\to \erre^2$, and let $J:\erre^2\to \erre^2$ be the map
$J:=Id-\nabla E$.
For every couple bounded Lipschitz map
$\boldsymbol{f}:=(f_1,f_2)$ with $f_i\in \rmW^{1,2}(X,\mathsf d,\mathfrak
m)$,
the components $g_i$ of  $\boldsymbol{g}:=J\circ \boldsymbol{f}$
belong to $ \rmW^{1,2}(X,\mathsf d,\mathfrak
m)$ and satisfy
\begin{equation}\label{disug_diff2}
  \vert\tD g_1\vert_\www^2(x)+\vert\tD g_2\vert_\www^2(x)\leq
  \abs{\tD f_1}_\www^2(x)+\abs{\tD f_2}_\www^2(x)\quad
  \text{for $\mathfrak m$-a.e.~}x\in X.
\end{equation} 
\end{lemma}
\begin{proof}
  Let us consider the case when $\mathfrak m(X)=+\infty$
  (the case of a finite measure is simper, and it follows by
  obvious modifications of the arguments below):
  notice that \eqref{eq:37} yields $\nabla E(0,0)=0$.

  Let us first notice that $|J(\boldsymbol f)|\le 2|\boldsymbol
  f|$
  so
  that for every $f_i\in L^2(X,\mathsf d,\mathfrak m)$
  the functions $g_i$ belong to $L^2(X,\mathsf d,\mathfrak m)$ as well.
  
  We first prove that
  \begin{equation}\label{disug_diff3}
  \vert\tD g_1\vert_\www^2(x)+\vert\tD g_2\vert_\www^2(x)\leq
  \abs{\tD f_1}^2(x)+\abs{\tD f_2}^2(x)\quad
  \text{for $\mathfrak m$-a.e.~}x\in X,
\end{equation}
whenever $f_1,f_2$ are bounded and Lipschitz and
$E$ is of class $C^{1,1}$. To this aim, it is sufficient to regularize
$\nabla E$ e.g.~by convolution with a family of smooth kernels
$\kappa_n:\mathbb R^2\to [0, +\infty)$, $n\in \mathbb N$ satisfying
\begin{displaymath}
  \kappa\in \mathrm C^\infty_c(\erre^2),\quad \kappa\ge0,\quad
  \kappa(-z)=\kappa(z),\
  \int_{\erre^2}\kappa\,\dd x=1,\quad
  \kappa_n(z):=n^2\kappa(n z)
  \quad z\in \erre^2.
\end{displaymath}
We then set
\begin{equation}
  \label{eq:29}
  \begin{aligned}
    E_n(x):={}&\int_{\erre^2}\Big(E(x-z)-x\cdot \nabla E(-z)\Big)\kappa_n(z)\,\dd z,\\
    \nabla
    E_n(x)={}&\int_{\erre^2}\Big(\nabla E(x-z)-\nabla E(-z)\Big)\kappa_n(z)\,\dd z,\\
    J_n(x):={}&x-\nabla E_n(x),\quad
    \boldsymbol g_n:=J_n\circ \boldsymbol f.
  \end{aligned}
\end{equation}
Applying Lemma \ref{lemma_crucial} we
get
\begin{displaymath}
  |\tD g_{n,1}|^2(x)+|\tD g_{n,2}|^2(x)\le
  |\tD f_1|^2(x)+|\tD f_2|^2(x).
\end{displaymath}
Since
\begin{align*}
  |J_n(x)-J(x)|&\le |\nabla E_n(x)-\nabla E(x)|
                 \\&\le
  \int_{\erre^2}\Big(\big|\nabla E(x-z)-\nabla E(x)\big|+\big|\nabla E(0)-\nabla
  E(-z)\big|\Big)\kappa_n(z)
                 \,\dd z
                 \\&\le 2 \int_{\erre^2}|z|\kappa_n(z)
                 \,\dd z=\frac 2n \int_{\erre^2}|z|\kappa(z)
                 \,\dd z,
\end{align*}
passing to the limit as $n\to\infty$ we have $g_{n,i}\to g_i$ in
$L^2(X,\mathfrak m)$; up to the extraction of
a suitable subsequence (not relabelled) we can also assume that
\begin{displaymath}
  |\tD g_{n,1}|\rightharpoonup G_1,\quad |\tD g_{n,2}|\rightharpoonup G_2\quad
  \text{weakly in $L^2(X,\mathfrak m)$ as $n\to\infty$}.
\end{displaymath}
We claim that
\begin{equation}
  \label{eq:32}
  G_1^2+G_2^2\le |\tD f_1|^2+|\tD f_2|^2
  \quad\text{$\mathfrak m$-a.e.~in $X$}.
\end{equation}
In fact, for an arbitrary measurable set $A\subset X$ 
we have
\begin{align*}
  \int_A \Big(G_1^2+G_2^2\Big)\,\dd\mathfrak m
  &=
    \lim_{n\to\infty}
    \int_{A} \Big(|\tD g_{n,1}|G_1+|\tD g_{n,2}|G_2\Big)\,\dd\mathfrak
    m
  \\&\le
  \Big(\int_A |\tD f_1|^2+|\tD f_2|^2\,\dd\mathfrak m\Big)^{1/2}
  \Big(\int_A \Big(G_1^2+G_2^2\Big)\,\dd\mathfrak m\Big)^{1/2}
\end{align*}
so that for every measurable set $A\subset X$
\begin{displaymath}
  \int_A \Big(G_1^2+G_2^2\Big)\,\dd\mathfrak m\le
  \int_A \Big(|\tD f_1|^2+|\tD f_2|^2\Big)\,\dd\mathfrak m.
\end{displaymath}
Since $|\tD g_i|_\www\le G_i$, \eqref{eq:32} yields \eqref{disug_diff3}.

\eqref{disug_diff2} then follows by \eqref{disug_diff3} by a similar
argument:
we select optimal sequences $(f_{i,n})_{n}$ of
bounded Lipschitz functions converging to $f_i$ in $L^2(X,\mathfrak
m)$
such that
\begin{displaymath}
  |\tD f_{i,n}|\to |\tD f_i|_\www\quad\text{in }L^2(X,\mathfrak m),\quad i=1,2,
\end{displaymath}
and we consider the corresponding sequences
$\boldsymbol g_n=J\circ \boldsymbol f_n$, converging to
$\boldsymbol g=J\circ\boldsymbol f$ in $L^2(X,\mathfrak m)$.
We then pass to the limit in the inequality
\begin{displaymath}
  \vert\tD g_{1,n}\vert_\www^2(x)+\vert\tD g_{2,n}\vert_\www^2(x)\leq
  \abs{\tD f_{1,n}}^2(x)+\abs{\tD f_{2,n}}^2(x)\quad
  \text{for $\mathfrak m$-a.e.~}x\in X.
\end{displaymath}
%
\end{proof}
\noindent Next lemma focuses on a useful property of the metric
Laplacian which relies on the estimate that we have just proved.
\nc
\begin{lemma}\label{Lap_new}
  If $f,g\in\mcD(\Delta)$ and $E:\erre^2\to\erre$ is a
  $\textnormal{C}^{1,1}$
  convex function satisfying \eqref{eq:37}, 
then \begin{equation}\label{Lapl3}
\int_{X}(\partial_1 E(f,g)\Delta f+\partial_2 E(f,g)\Delta g)\,\dd\gm\leq 0.
\end{equation}
\end{lemma}
\begin{proof}
  It is not restrictive to assume that $E$ is $1$-Lipschitz.
  As we observed in the proof of Lemma \ref{rmk_relax},
  we also note that $\partial_i E(f,g)$ belongs to $L^2(X,\gm)$,
  since when $\mathfrak m(X)=+\infty$ \eqref{eq:37} yields
  $\partial_i E(0,0)$;
  therefore the integral in \eqref{Lapl3} is well defined. 
Recall that \[ 
l\in\partial\mathsf{Ch}(\varphi)\quad \Longleftrightarrow\quad \int_{X}l(\psi-\varphi)\,\dd\gm\leq \mathsf{Ch}(\psi)-\mathsf{Ch}(\varphi)\quad\text{for every }\psi\in L^2(X,\gm)
\]   
and that $-\Delta \varphi\in\partial\mathsf{Ch}(\varphi)$. 
Hence taking in our case $\varphi=f$ and $\psi=f-\partial_1 E(f,g)$ we get \[\begin{split} 
\int_{X}\partial_1 E(f,g)\Delta f\,\dd\gm=\int_{X}-\Delta f(\psi-\varphi)\,\dd\gm&\leq
 \mathsf{Ch}(\psi)-\mathsf{Ch}(\varphi)\\&=\mathsf{Ch}(f-\partial_1 E(f,g))-\mathsf{Ch}(f),
\end{split}\]
and similarly \[ 
\int_{X}\partial_2 E(f,g)\Delta g\,\dd\gm\leq \mathsf{Ch}(g-\partial_2 E(f,g))-\mathsf{Ch}(g).
\]
By definition of the Cheeger functional and Lemma \ref{rmk_relax} we obtain \eqref{Lapl3}.
\end{proof}
\noindent
With the previously developed tools we can conclude
the proof of Theorem \ref{contr_convex}.\begin{proof}
Let us set $f_t=\sfP_t f$ and $g_t=\sfP_t g$. Assume first  that $E$
is $\textnormal{C}^{1,1}$ with
Lipschitz gradient $\nabla E$, so that $E$ has at most quadratic
growth. Recalling that $t\mapsto f_t,g_t$ are differentiable as
$L^2$-valued maps, we get
\begin{align*}
  \frac{\dd}{\dd t}\mathscr{E}(f_t,g_t)
  &=\int_{X}\frac{\dd}{\dd t}E( f_t,g_t)\,\dd\gm
    =\Big(\int_{X}\partial_1E(f_t,g_t)\Delta_{\mathsf{d},\gm}f_t+\partial_2E(f_t,g_t)\Delta_{\mathsf{d},\gm}g_t
    \Big)\,\dd\gm
  \\&=
  \int_{X}
  \Big(\partial_1E(f_t,g_t)\Delta_{\mathsf{d},\gm}f_t+\partial_2E(f_t,g_t)\Delta_{\mathsf{d},\gm}g_t
  \Big)\,\dd\gm\leq 0
\end{align*}
thanks to \eqref{Lapl3}.
We thus obtain
\begin{equation}
\mathscr{E}(\sfP_t f,\sfP_t g)\leq \mathscr{E}(f,g)\quad\text{for
  every }t\ge0.\label{eq:31}
\end{equation}
In the general case, we apply \eqref{eq:31} to
the functional $\mathscr E_\lambda$ associated to the Yosida approximation
$E_\lambda$ of $E$,
\begin{equation}
  \label{eq:38}
  E_\lambda(r,s):=\inf_{(r',s')\in \mathbb
    R^2}\frac1{2\lambda}\Big((r'-r)^2+(s'-s)^2\Big)+E(r',s')\quad
  (r,s)\in \mathbb R^2,\ \lambda>0.
\end{equation}
It is well known \cite{Brezis73} that $E_\lambda$ is convex of class $C^{1,1}$ with
Lipschitz gradient $\nabla E_\lambda$;
moreover, if \eqref{eq:37} holds, then also $E_\lambda$ is nonnegative
and it is immediate to check from \eqref{eq:38} that
$E_\lambda(0,0)=0$. Since $E_\lambda\le E$,
\eqref{eq:31} then yields
\begin{displaymath}
  \mathscr{E}_\lambda(\sfP_t f,\sfP_t g)\leq \mathscr{E}_\lambda(f,g)
  \le \mathscr E(f,g)\quad\text{for
  every }t\ge0,\ \lambda>0.
\end{displaymath}
We can eventually pass to the limit as
$\lambda\downarrow0 $ and applying Beppo Levi monotone convergence
theorem,
since $E_\lambda(r,s)\uparrow E(r,s)$ as $\lambda\downarrow0$
for every $r,s\in \mathbb R^2$.
\nc
\end{proof}
\noindent A few particular cases follow as corollaries of the main
result. The first one states the contraction in the Hellinger metric
 for measures which are absolutely continuous w.r.t.~$\gm$:
with a slight abuse of notation,
for every $f,g\in L^1(X,\gm)$, $f,g\ge0$, we will set
\begin{displaymath}
  \hed^p_p(f,g):=\hed_p(f\gm,g\gm)=\int_X\Big|f^{1/p}-g^{1/p}\Big|^p\,\dd\gm.
\end{displaymath}
\begin{corollary}\label{cor:contrHell}
  For every nonnegative $f,g\in L^1(X,\gm)$
  we have
  \begin{equation}\label{eq:42}
    \hed_p(\sfP_t f,\sfP_t g)\leq \hed_p(f,g)\quad\text{for every }t\ge0.
\end{equation} 
\end{corollary}
\begin{proof}
  It is sufficient to prove \eqref{eq:42} for every couple of
  nonnegative functions $f,g\in L^1\cap L^2(X,\gm)$
  and then argue by approximation using \eqref{eq:41} for $p=1$.
  We can then apply Theorem \ref{contr_convex}
  with the function
  $E:\mathbb{R}^2\rightarrow \mathbb{R}\cup\{+\infty\}$ given by
  \begin{displaymath}
    E(r,s):=
    \begin{cases}
      \abs{r^{1/p}-s^{1/p}}^p&\text{if }r,s\ge0,\\
      +\infty&\text{otherwise}.
    \end{cases}
  \end{displaymath}
\end{proof}
\noindent More generally, the same contraction result holds true for
any Csisz\`ar divergence;
recalling the discussion of Section \ref{sub:entropy}
and keeping the same notation of \eqref{eq:5}, \eqref{eq:3},
\eqref{eq:4} and Definition \ref{def:Csiszar}, we first set
\begin{displaymath}
  \mathscr F(f|g):=
  \mathscr{F}(f\gm\mid g\gm)=
  \int_X H(f,g)\,\dd\gm\quad
  \text{for every nonnegative }f,g\in L^1(X,\gm).
\end{displaymath}
\begin{corollary}
  Let $\mathscr{F}$ be a Csisz\`ar divergence as in
  Definition \ref{def:Csiszar}.
  Then, for every nonnegative $f,g\in L^1(X,\gm)$,
  \begin{displaymath}
    \mathscr{F}(\sfP_t f\mid \sfP_t g )\leq \mathscr{F}(f\mid g ).
\end{displaymath}
\end{corollary}
\begin{proof}
  Recalling \eqref{eq:11}, it is sufficient to apply Theorem
  \ref{contr_convex}
  to the integral functional associated to the function $H$ of
  \eqref{eq:3},
  extended to $+\infty$ if $r<0$ or $s<0$.
  
\end{proof}
\nc

\section{Regularizing properties of the Heat flow in $\RCD$ metric
  measure
spaces}
\label{sec:regularization}
\GGG\noindent In the previous section we have shown contraction results
involving
convex functionals
and metric heat flows in metric measure spaces, thus covering the case
of nonlinear flows in Finsler-like geometries.

In the linear case, the Hellinger contraction \eqref{eq:42}
can also be proved by a different approach,
based on the dual dynamic formulation of the Hellinger distance
that we have discussed in \ref{prop:helldyndual}.
We first explain this technique in the simple case of a submarkovian
operator $\sfP$ on the set of bounded measurable functions and we
will
then show how to extend this approach to prove new regularization
results for the Heat semigroup in $\RCD$ metric measure spaces.

\subsection{Hellinger contraction for submarkovian operators}
Let $(\Omega,\mathscr B)$ be a measurable space and
let $\sfP:\mathrm B_b(\Omega)\to\mathrm B_b(\Omega)$
be a linear submarkovian operator
\cite[Chap.~IX, Sect. 1]{Dellacherie-Meyer88}:
this means that for every bounded measurable maps $f,f_n\in \mathrm B_b(\Omega)$
\begin{subequations}
  \begin{align}
    \label{eq:44}
    0\le f\le 1\quad \Rightarrow\quad 0\le \sfP f\le1,\\
    \label{eq:45}
    f_n\ge 0,\ f_n\downarrow 0\text{ as }n\to\infty\quad\Rightarrow\quad
    \sfP f_n\downarrow0,
  \end{align}
\end{subequations}
where convergence in \eqref{eq:45} has to be intented pointwise
everywhere.
Notice that for every $x,y\in \Omega$
\begin{align*}
  0&\le \sfP\Big((f-\sfP f(y))^2\Big)(x)=
     \sfP f^2(x)-2\sfP f(y)\sfP f(x)+(\sfP f(y))^2\sfP 1
     \\&\le
  \sfP f^2(x)-\big(\sfP f(x)\big)^2+\big(\sfP f(x)-\sfP f(y)\big)^2
\end{align*}
so that choosing $x=y$ we get the Jensen's inequality
\begin{equation}\label{eq:47}
  (\sfP f)^2\le \sfP f^2.
\end{equation}
We can define the adjoint operator $\sfP^*$ acting on $\calM(\Omega)$
by the formula
\begin{displaymath}
  \int_\Omega f\,\dd \sfP^*\mu:=\int_\Omega \sfP f\,\dd\mu\quad
  \text{for every }f\in \mathrm B_b(\Omega).
\end{displaymath}
The next result could also be derived by a more refined Jensen inequality for
submarkovian operator. Here we want to highlight the role of the dual
dynamic
point of view.
\begin{proposition}\label{contr_Hell_thm}
  Let $(\Omega,\mathscr B)$ be a measure space and
  let $\sfP$ be a linear submarkovian operator in $\mathrm
  B_b(\Omega)$.
  Then, for any $\mu_0,\mu_1\in \calM(\Omega)$
  \begin{equation}\label{contr_Hell} 
    \hed_2(\sfP^* \mu_0, \sfP^* \mu_1)\leq \hed_2(\mu_0,\mu_1).
\end{equation}
\end{proposition}

\begin{proof}
Let us consider $(\zeta_s)_{s\in [0,1]}\in \text{C}^1([0,1],\mathrm B_b(\Omega))$ a solution of \begin{equation}\label{eq_dual_hell} 
\partial_s\zeta_s+\zeta_s^{2}\leq 0\quad\text{ in }\Omega\times [0,1].
\end{equation}
We apply the map $\sfP$ to this solution; since the linear map $\sfP$ is continuous
with respect to the supremum norm in $\mathrm B_b(\Omega)$,
$(\sfP\zeta_s)_s\in\textnormal{C}^1([0,1],\mathrm B_b(\Omega))$.
Moreover, from \eqref{eq:47}  applied to $\zeta_s$
it follows that $s\to \sfP\zeta_s$ is also a subsolution to
\eqref{eq_dual_hell}:
\[ 
\partial_s \sfP\zeta_s+(\sfP\zeta_s)^2\leq
\sfP\partial_s\zeta_s+\sfP(\zeta_{s}^{2})
=\sfP(\partial_s\zeta_s+\zeta_s^2)\leq 0,
\]
since $\sfP$ is positivity preserving.
Then, recalling the formulation \eqref{hell_dyn_dual} of the Hellinger
distance, we have
\[
  \int_{\Omega}\zeta_1\,\dd(\sfP^*\mu_1)-\int_{\Omega}\zeta_0\,\dd(\sfP^*\mu_0)
  =\int_{\Omega}\sfP\zeta_1\,\dd \mu_1-\int_{\Omega}\sfP\zeta_0\,\dd
  \mu_0
  \leq\hed^2(\mu_0,\mu_1).
\]
Taking the supremum of the left hand side with respect to all the subsolutions of
\eqref{eq_dual_hell} and applying \eqref{hell_dyn_dual} once more, 
we eventually get \eqref{contr_Hell}.
\end{proof}
\begin{remark}
  The same argument combined with the $p$-Jensen inequality for $\sfP$
  yields
  \[ 
\hed_p(\sfP^*\mu_0, \sfP\mu_1)\leq \hed_p(\mu_0,\mu_1),
\]
for every $p\in [1,+\infty)$.
The proof can also be extended to submarkovian operators in
$L^1(\Omega,\gm)$ with respect
to a given reference measure $\gm\in \calM(\Omega)$,
obtaining in this case an Hellinger estimate for measures absolutely
continuous w.r.t.~$\gm$.
\end{remark}
\subsection{Regularization $\sfW_p$- $\mathsf{He}_p$ for $p\in [1,2]$}

Let us now focus on the regularization estimates for the particular
class of Markovian operators provided by the
heat semigroup $(\sfP_t)_{t\geq 0}$ in
a metric measure space $(X,\sfd,\mathfrak{m})$ satisfying the
$\RCD(K,\infty)$ condition.
%
Since $(\sfP_t)_{t\ge0}$
maps $\mathrm C_b(X)$ into $\mathrm C_b(X)$,
we can use \eqref{eq:51} to define
the adjoint heat semigroup $(\sfP^*_{t})_{t\geq 0}$ on arbitrary
positive and finite
measure of $\calM(X)$ (see \cite[Section 3.2]{AGS15} for more details).

\begin{theorem}\label{reg_1_th}
Let $(X,\mathsf{d},\gm)$ be a $\RCD(K,\infty)$ metric measure space
and
\GGG $p\in [1,2]$.
\nc
Then, for every $\mu_0,\mu_1\in\calP_p(X)$ \begin{equation}\label{reg_1}
\hed_p(\sfP_t^*\mu_0,\sfP_t^*\mu_1)\leq \frac{1}{p (R_{K}(t))^{1/2}}\sfW_p(\mu_0,\mu_1) \qquad \textnormal{for all }t>0,
\end{equation}
where $R_{K}$ has been defined in \eqref{eq:66}.
\end{theorem}



\begin{proof}
Let us set $\mathscr{C}^1(\mathrm B_b):=\textnormal{C}^1([0,1],\mathrm
B_b(X))$ and
$\mathscr{C}^1(\textnormal{Lip}_b):=\textnormal{C}^1([0,1],\textnormal{Lip}_b(X))$
to shorten the notation.
We will consider the case $p>1$; the case $p=1$ follows directly from
\eqref{eq:82}
by using the dual characterization of the Kantorovich distance
$\sfW_1$, or by approximating $\mu_0,\mu_1$ by measures with bounded
support (thus in $\calP_{p_0}(X)$ for every $p_0\in [1,\infty[$)
and then passing to the limit in \eqref{reg_1} as
$p\downarrow1$.

The dual dynamic formulations \eqref{hell_dyn_dual} and
\eqref{dyn_Wass}
(recall that a $\RCD$ space is a length space)
we know \[ 
\hed_p^{p}(\mu_0,\mu_1)=\sup \Big\{\int_{\Omega}\zeta_1\,\dd\mu_1-\int_{\Omega}\zeta_0\,\dd\mu_0:\text{ }\zeta\in\mathscr{C}^1(\mathrm B_b),\quad\partial_s\zeta_s+\frac{\zeta_s^q}{q-1} \leq 0\Big\}
\] and \[ 
  \frac 1p
  \sfW_p^p(\mu_0,\mu_1)=\sup\Big\{\int_X\zeta_1\,\dd\mu_1-\int_X\zeta_0\,\dd\mu_0:\,\,\,\zeta\in \mathscr{C}^1(\textnormal{Lip}_b),\quad\partial_s\zeta_s+\frac{1}{q}\abs{\tD \zeta_s}^q\leq 0 \Big\}.
\]
A simple rescaling argument, replacing $\zeta$ by $\frac pa \zeta$,
shows that  for $a>0$
\begin{equation}
\frac{1}{a}\sfW_{p}^{p}(\mu_0,\mu_1)=\sup\Big\{\int_{X}\zeta_1\,\dd\mu_1-\int_{X}\zeta_0\,\dd\mu_0:\text{
  }\zeta\in\mathscr{C}^1(\textnormal{Lip}_b),
  \ \partial_s\zeta_s+\frac{a^{q-1}}{q\,p^{q-1}}\abs{ \tD
    \zeta_s}^q\leq 0\Big\}.
  \label{eq:49}
\end{equation}
%
Now, take $\zeta\in \textnormal{C}^1([0,1], \mathrm B_b(X))$ such that
$\partial_s\zeta_s+\frac{1}{q-1}\zeta_s^q\leq 0$.
We apply the order preserving semigroup $(\sfP_t)_{t\geq 0}$ to $(\zeta_s)_s$ and we
get 
\begin{equation}
\label{sost}
\partial_s\sfP_t\zeta_s+\frac{1}{q-1}\sfP_t\zeta_s^q\leq
0.
\end{equation}
The Lipschitz regularization property
stated in Theorem \ref{thm:cond_equivalenti} ensures that
$\sfP_t(\zeta_s)\in\tnLip_b(X)$ and that it satisfies the refined Bakry-Emery
condition \eqref{key_point}, where we neglect the last negative term:
\beq\label{da_elevare} 
R_{K}(t)\abs{\tD \sfP_t\zeta_s}^2\leq \sfP_t\zeta_s^2 \qquad\textnormal{in }X.
\eeq
Since $p\in (1,2]$ , the conjugate $q$ is in $[2,+\infty)$ and hence $q/2\geq 1$. Taking the power $q/2$ in \eqref{da_elevare} and using Jensen's inequality we obtain \[ 
(R_{K}(t))^{\frac{q}{2}}\abs{\tD \sfP_t\zeta_s}^q\leq
\big(\sfP_t(\zeta_s)\big)^{q/2}\le \sfP_t(\zeta_s^q).
\]
The combination of this inequality and \eqref{sost} yields \[ 
\partial_s\sfP_t\zeta_s+\frac{p^{q-1}q(R_{K}(t))^{\frac{q}{2}}}{q-1}\frac{\abs{\tD \sfP_t\zeta_s}^q}{p^{q-1}q}=\partial_s\sfP_t\zeta_s+p^q(R_{K}(t))^{\frac{q}{2}}\frac{\abs{\tD \sfP_t\zeta_s}^q}{p^{q-1}q}\leq 0 
\]
which shows that $\tilde\zeta_s:=\sfP_t\zeta_s$
is a subsolution of the Hamilton-Jacobi equation as in \eqref{eq:49}
with the time-and-curvature dependent weight
\begin{equation}
a(t):=\Big(p^q(R_{K}(t))^{\frac{q}{2}}\Big)^{1/{q-1}}=
p^p (R_K(t))^{p/2}.\label{eq:50}
\end{equation}
All these facts lead to \begin{align*} \displaystyle
&\int_X\zeta_1\,\dd(\sfP^*_t\mu_1)-\int_X\zeta_0\,\dd(\sfP^*_t\mu_0)
                          =\int_X\sfP_t\zeta_1\,\dd\mu_1-\int_X\sfP_t\zeta_0\,\dd\mu_0\leq
                          \frac1{a(t)}
                          \sfW_p^p(\mu_0,\mu_1).
\end{align*}
Thus, taking the supremum over all the subsolutions to $\partial_s\zeta_s+\frac{1}{q-1}\zeta_s^q\leq 0$ we conclude \[ 
\hed_p^p(\sfP_t^*\mu_0,\sfP^*_t\mu_1)\leq \frac{1}{a(t)}\sfW_p^p(\mu_0,\mu_1)
\]
where $a(t)$ as in \eqref{eq:50},
which yields \eqref{reg_1}.
\end{proof}
As a byproduct, when $K\ge0$, we obtain an precise decay rate for
the asymptotic behaviour of $\sfP_t^*$.
\begin{corollary}
  \label{cor:asymptotic}
  Let $(X,\mathsf{d},\gm)$ be a $\RCD(K,\infty)$ metric measure space
  with $K\ge0$ and let $\gm\in \calP_p(X)$, $p\in [1,2]$.
  For every $\mu_0\in \calP_p(X)$ we have
  \begin{equation}
    \label{eq:84}
    \hed_p(\sfP_t^*\mu_0,\gm)\le
    \frac1{p(R_K(t))^{1/2}}\sfW_p(\mu_0,\gm)\quad
    \text{for every }t>0.
  \end{equation}
\end{corollary}
In the case $p=2$ and $K>0$ it is interesting to compare \eqref{eq:84}
with
the well known exponential decay rates of the logarithmic entropy
and of the Wasserstein distance
\begin{equation}
  \label{eq:28}
  \KL(\sfP_t^*\mu_0|\gm)\le \mathrm e^{-2K t}\KL(\mu_0|\gm),\quad
  \sfW_2(\sfP_t^*\mu_0,\gm)\le \mathrm e^{-K t}\sfW_2(\mu_0,\gm)
\end{equation}
which follow by the $K$-geodesic convexity of the $\KL$ functional
in $\mathsf{CD}(K,\infty)$ spaces.
In particular, recalling \eqref{eq:83},
the first estimate of \eqref{eq:28} provides
\begin{displaymath}
  \hed_2(\sfP_t^*\mu_0|\gm)\le \mathrm e^{-K t}\KL(\mu_0|\gm)
\end{displaymath}
which exhibits the same exponential behaviour of \eqref{eq:84};
however, \eqref{eq:84} only requires $\mu_0\in \cP_2(X)$.
 \subsection{{Regularization $\mathsf{He}_2$-$\mathsf{HK}$}}
 \noindent With a similar argument we prove that the Hellinger
 distance at time $t$ can be estimated from above by the weighted
 Hellinger-Kantorovich distance $\HK_\alpha$,
 in which the parameter $\alpha$ acts on the
 transport part of the distance with a time-dependent factor and
 does not affect the reaction part.
 Note that this embodies a natural combination of the Hellinger-Kantorovich estimate above and the Hellinger contraction that we proved in Proposition \ref{contr_Hell_thm}.
 \begin{theorem}\label{cor_stima}
   Let $(X,\mathsf{d},\gm)$ be a $\RCD(K,\infty)$ metric measure space.
   For every
   $\mu_0,\mu_1\in\calM(X)$
   \begin{equation}\label{stima_imp_HK2}
     \hed_2(\sfP_{t}^{\ast}\mu_0,\sfP_{t}^{\ast}\mu_1)\leq
     \HK_{\alpha(t)}(\mu_0,\mu_1),
     \qquad
     \text{$\alpha(t)=4R_K(t)$ as defined in \eqref{eq:66}}.
 \end{equation}
\end{theorem}
\begin{proof}
As in the previous proof, we set
$\mathscr{C}^1(\mathrm B_b):=\textnormal{C}^1([0,1],\mathrm
B_b(\Omega))$ and
$\mathscr{C}^1(\textnormal{Lip}_b):=\textnormal{C}^1([0,1],\textnormal{Lip}_b(\Omega))$
and 
we recall that
\[ \hed_2^{2}(\mu_0,\mu_1)=\sup\big\{ \int_{X}\zeta_1\,\dd\mu_1-\int_{X}\zeta_0\,\dd\mu_0:\text{ }\zeta\in\mathscr{C}^1(\mathrm B_b),\quad\partial_s\zeta_s+\zeta_s^2\leq 0\big\}.\]
and that
\eqref{HK_weight}
\begin{equation}
 \begin{aligned}
   \HK_{\alpha}^2(\mu_0,\mu_1)= \sup\Big\{&
   \int_{X} \zeta_1\,\dd\mu_1 -\int_{X} \zeta_0\,\dd\mu_0\,:\text{ }\zeta\in \mathscr{C}^1(\textnormal{Lip}_b), \\
 &\partial_s\zeta_s(x)+\frac{\alpha}{4}\vert D_X
 \zeta_s\vert^2(x)+\zeta_s^2\leq 0\Big\}\end{aligned}\label{eq:53}
\end{equation}
We consider a solution $\zeta\in\mathscr{C}^1(\mathrm B_b)$
of $\partial_s\zeta_s+\zeta_s^2\leq 0$ and
we apply the linear operator $\sfP_t$, $t>0$, obtaining \[ 
\partial_s\sfP_t\zeta_s+\sfP_{t}(\zeta_s)^2\leq 0.
\]
Theorem \ref{thm:cond_equivalenti} ensures that $\sfP_t\zeta_s$ is Lipschitz and
satisfies
\[ 
R_{K}(t)\abs{ \tD \sfP_{t}\zeta_s}^2+(\sfP_t \zeta_s)^2\leq \sfP_{t}(\zeta_s^2)
\]
so that  \[ 
\partial_s \sfP_t\zeta_s+R_{K}(t)\abs{\tD \sfP_{t}\zeta_s}^2+(\sfP_t\zeta_s)^2\leq 0;
\]
this inequality corresponds to the subsolutions of Hamilton-Jacobi equation
in \eqref{eq:53} weighted with $\alpha=4R_{K}(t)=\alpha(t)$.
Therefore \[ 
\int_{X}\zeta_1\,\dd(\sfP^{\ast}_{t}\mu_1)-\int_{X}\zeta_0\,\dd(\sfP_{t}^{\ast}\mu_0)=\int_{X}\sfP_{t}\zeta_1\,\dd\mu_1-\int_{X}\sfP_t\zeta_0\,\dd\mu_0\leq \HK^{2}_{\alpha(t)}(\mu_0,\mu_1),
\]
and taking the supremum with respect to the subsolutions to
$\partial_s\zeta_s+\zeta_s^2\leq 0$ we get
\eqref{stima_imp_HK2}.
\end{proof}
It is worth noticing that 
\eqref{stima_imp_HK2}
yields the pure Hellinger contraction estimate \eqref{eq:42}
thanks to \eqref{eq:67}.
  Similarly, choosing $\mu_0,\mu_1\in \calP_2(X)$ and
  applying \eqref{eq:79} one recovers \eqref{reg_1} in the case $p=2$.
\bibliography{refe}{}
\bibliographystyle{abbrv}
\end{document}